\documentclass[12pt,reqno]{amsart}
\usepackage{amsmath,amsfonts,amssymb,amscd,amsthm,amsbsy,epsf}
\usepackage[dvips]{graphicx}
\usepackage{comment}
\usepackage[dvips]{color}
\numberwithin{equation}{section}
\newtheorem{theorem}{Theorem}
\newtheorem{meta-thm}[theorem]{Meta-Theorem}
\newtheorem{format}[theorem]{Theorem Format}
\newtheorem{lemma}[theorem]{Lemma}

\newtheorem{proposition}[theorem]{Proposition}
\newtheorem{algorithm}[theorem]{Algorithm}
\newtheorem{remark}[theorem]{Remark}
\newtheorem{definition}[theorem]{Definition}

\newcommand{\noaverage}[1]{({#1})^0}

\newcommand\beq[1]{ \begin{equation}\label{#1} }
\newcommand{\eeq}{ \end{equation} }

\newcommand\beqa[1]{ \begin{eqnarray} \label{#1}}
\newcommand{\eeqa}{ \end{eqnarray} }
\newcommand{\beqano}{\begin{eqnarray*} }
\newcommand{\eeqano}{ \end{eqnarray*} }
\newcommand\equ[1]{{\rm (\ref{#1})}}

\def\ep{\varepsilon}

\def\dist{\operatorname{dist}}

\def\Id{\operatorname{Id}}

\def\Int{\operatorname{Int}}
\def\Range{\operatorname{Range}}
\def\A{{\mathcal A}}
\def\C{{\mathcal C}}
\def\D{{\mathcal D}}

\def\E{{\mathcal E}}
\def\M{{\mathcal M}}
\def\complex{{\mathbb C}}
\def\integer{{\mathbb Z}}
\def\nat{{\mathbb N}}
\def\E{{\mathcal E}}

\def\real{{\mathbb R}}
\def\torus{{\mathbb T}}
\def\zed{{\mathbb Z}}

\begin{document}

\title[KAM estimates for the dissipative standard map]
{KAM estimates for the dissipative standard map}

\author[R. Calleja]{Renato C.  Calleja}
\address{Department of Mathematics and Mechanics, IIMAS, National
  Autonomous University of Mexico (UNAM), Apdo. Postal 20-126,
  C.P. 01000, Mexico D.F., Mexico}
\email{calleja@mym.iimas.unam.mx}

\author[A. Celletti]{Alessandra Celletti}
\address{
Department of Mathematics, University of Rome Tor Vergata, Via della Ricerca Scientifica 1,
00133 Rome (Italy)}
\email{celletti@mat.uniroma2.it}

\author[R. de la Llave]{Rafael de la Llave}
\address{
School of Mathematics,
Georgia Institute of Technology,
686 Cherry St., Atlanta GA. 30332-1160 }
\email{rafael.delallave@math.gatech.edu}

\thanks{A.C. was partially supported by GNFM-INdAM, EU-ITN Stardust-R, MIUR-PRIN 20178CJA2B ``New Frontiers
of Celestial Mechanics: theory and Applications'' and acknowledges
the MIUR Excellence Department Project awarded to the Department
of Mathematics, University of Rome Tor Vergata, CUP
E83C18000100006. R.L. was partially supported by NSF grant
DMS-1800241. R.C. was partially supported by UNAM-DGAPA PAPIIT
projects IA102818 \& IN 101020.
Part of this material is based upon work supported by
the National Science Foundation under Grant No. DMS-1440140 while
the authors were in residence at the Mathematical Sciences
Research Institute in Berkeley, California, during the Fall 2018
semester.}

\baselineskip=18pt              


\begin{abstract}
From the beginning of KAM theory, it was realized that its
applicability to realistic problems depended on developing
quantitative estimates on the sizes of the perturbations allowed.
In this paper we present results on the existence of
quasi-periodic solutions  for
conformally symplectic systems in non-perturbative
regimes. We recall that, for conformally symplectic
systems, finding the solution requires also to find
a \emph{drift parameter}.  We present a proof on the existence of
solutions for
values of the parameters which agree with more than three
figures with  the numerically conjectured optimal values.

The first step of the strategy is to establish a very
explicit quantitative theorem in an a-posteriori format.
We recall that in numerical analysis, an a-posteriori theorem
assumes the existence of
an approximate solution, which satisfies an invariance
equation up to an error which is small enough with respect to
explicit condition numbers, and then concludes the existence of
a solution. In the case of conformally symplectic systems,
an a-posteriori theorem was proved in \cite{CallejaCL11}.
Our first task is to make all the constants fully explicit.

We emphasize  that our result allows to conclude the existence of
the true solution by verifying mainly that the approximate solution
satisfies the equation up to a small error and that some
condition numbers are finite. The method used
to produce the approximate solution does not need to be examined.

The second step in the strategy is to produce numerically very
accurate solutions in a concrete problem. We have implemented the
algorithm indicated in \cite{CallejaCL11}  in a model problem,
widely considered in the literature; we constructed numerically
very accurate solutions of the invariance
equations (discretizations with $2^{18}$ Fourier coefficients, each
one computed with 100 digits of precision). From the point of view
of rigorous mathematics, we note that the first step is a fully
rigorous theorem, the second step is a high precision calculation
which produces an impact for the theorem in the first part.

The third and final step is to present
a numerical verification of the hypotheses of the theorem
stated in the first part on the numerical solutions presented in the second part.
Using these estimates we would conclude  the existence of
tori  for  certain values of the drift parameter.
The perturbation parameters we can consider coincide with
more than 3 significant figures  with the values conjectured as
optimal by numerical experiments.

The verification of the estimates presented here
is not completely rigorous since we
do not control the round-off error.
Nevertheless, running with different precision shows very little difference
in the results.
Given the high precision of
the calculation and the simplicity of the estimates, this does
not seem to affect the results. A full verification
should be done implementing interval arithmetic.

We make available the approximate solutions, the highly efficient
algorithms to generate them (incorporating high precision based on the MPFR library)
and the routines used to verify the applicability of the theorem.
\end{abstract}


\maketitle


\section{Introduction}\label{sec:intro}

The goal of this paper is to develop a methodology to
compute efficiently and to verify rigorously the existence of quasi-periodic solutions
in concrete systems
(compare with \cite{MR912758,Celletti90I,Celletti90II,MR2684063,CellettiC95,CellettiC97,CellettiC07,CellettiC09,
FHL,LGS2,LlaveR90,Locatelli,LGS1,StefanelliL15}).

The celebrated KAM theory, started in
\cite{Kolmogorov54,Arnold63a,Moser62},
solved the outstanding problem of establishing the
persistence of quasi-periodic orbits under small perturbations.
An important motivation was represented by problems in celestial mechanics
(\cite{Arnold63b}).
By now, KAM theory has developed into a very useful paradigm.
Surveys of KAM theory and its applications are:
\cite{Arnold63b, Moser66a, Moser66a, Moser66b, Bost86, Yoccoz92b, Llave01c, CellettiC95, Fejoz17}.

At the beginning of the theory, the quantitative requirements
for applicability led to unrealistic smallness estimates.
In a well known calculation (\cite{Henon66}), M. H\'enon
made a preliminary study of the parameters required
to apply to the three-body problem (\cite{Arnold63a}) and obtained that the
small parameter (representing the Jupiter--Sun mass--ratio) should be smaller than
$10^{-48}$, whereas of course, the real value for Jupiter is about
$10^{-3}$.
Discouraged by this result, the often quoted conclusion of \cite{Henon66} was
that\footnote{``It does not seem that these theorems, though having a great theoretical interest, can be applied,
\bf in their present state\rm, to practical problems" \cite{Henon66}.}

\emph{``Ainsi, ces th\'eor\`emes, bien que d'un tr\`es grand int\'er\^et th\'eorique, ne semblent
pas pouvoir \textbf{en leur \'etat actuel} \^etre appliqu\'es \'a des probl\`emes pratiques''.}

Even if the statement of
\cite{Henon66} is perfectly correct as stated, removing the words
we have set in bold (as it is often done),
 one obtains a statement invalid 50 years after
the original statement.

It is also true that the first attempts to study the
problem numerically were disappointing. The persistence of
quasi-periodic solutions indeed depends
on rather higher regularity of the perturbation
(the smoothness requirements of some
versions of KAM theory are optimal, \cite{Herman83, ChengL13, MramorR14})
and attempts based on low regularity
discretizations such as finite elements were discouraging
(\cite{BraessZ82}).
Furthermore, unless one is careful, one can be misled by spurious solutions.
It is also true that many of the original proofs were based on transformation theory, which
is difficult to implement numerically (one needs to deal
with functions of a high number of variables and impose
that they satisfy geometric constraints).   More successful studies
such as \cite{Greene79, Chirikov79, Aubry79,  Percival74, Percival79, MacKay82}
were based on indirect methods for very specific systems
and were challenged because of being indirect.

By the late 70's it was folklore belief that the estimates of
KAM theory were essentially optimal (the optimal steps for
one step were  known to good approximation,
and one could hope that, by some Baire
category argument, one could find systems that saturate the bounds
to all steps).

By now, the situation has changed drastically.  There are general
bounds based on different schemes (\cite{Brjuno72, Yoccoz95, Llave83, Herman83}),
which lead to substantially better bounds. In practical applications,
one is interested in concrete systems, not on generic ones
and there has been also progress in obtaining estimates in
some specific systems. Later on, many more situations
leading to better bounds were found, see \cite{Celletti90I, Celletti90II, CellettiC95, CellettiC97}.

More related to the present paper, in recent times
there has been a rapid development in proofs of
KAM theorems in the \emph{``a-posteriori''} format
common in numerical analysis.
We recall that an a-posteriori theorem in numerical
analysis is a theorem of the following format.

\begin{format}
\label{aposteriori}
Let $\mathcal{X}_0 \subset \mathcal{X}_1$ be Banach spaces and
$\mathcal{U}\subset \mathcal{X}_0$ an open set. Consider the map
\[\mathcal{F}: \mathcal{U} \subset \mathcal{X}_0 \rightarrow \mathcal{X}_0\ ; \]
assume that there are functionals
$m_1,...,m_n:\mathcal{U} \rightarrow \mathbb{R}^+$ for some $x_0 \in \mathcal{X}_0$, such that:
\begin{enumerate}
\item $\|\mathcal{F}(x_0)\|_{\mathcal{X}_0}< \varepsilon$ for some $\varepsilon\in\real$;
\item $m_1(x_0)\leq M_1, ..., m_n(x_0)\leq M_n$ for some condition numbers $M_1$, ..., $M_n$;
\item $\varepsilon \leq \varepsilon^*(M_1, ... , M_n)$, where $\varepsilon^*$ is an
explicit function of the condition numbers.
\end{enumerate}
Then, there exists an $x^* \in X_1$ such that $\mathcal{F}(x^*) = 0$ and
$\|x_0 - x^*\|_{\mathcal{X}_1} \leq C_{M_1, ..., M_n} \varepsilon$
for some positive constant $C_{M_1,...,M_n}$.
\end{format}

Of course, to obtain the statement of a  theorem in
the Format~\ref{aposteriori}, one has to specify
all the ingredients,
$\mathcal{X}_0$,  $\mathcal{X}_1$,  $\mathcal{F}$, $m_1,\ldots, m_n$,
the function $\varepsilon^*$ and provide a proof; Theorems of this
form are very common e.g. in finite elements theory (\cite{Oden}) or in linear
algebra.

As it turns out, one can formulate several
 KAM theorems in this format.
One needs to choose an appropriate functional $\mathcal{F}$
whose zeros imply the  existence of quasi-periodic solutions
(in such applications $x$ is an embedding that belongs to
a suitable space of functions, see \ref{sec:norms}.)

Notice
that in contrast with other more customary versions of KAM
theory, this formulation does not involve that we
are considering a system close to integrable and it
does not require any global assumption on the map, but only some
functionals  evaluated in the approximate solution. In the
problems considered in this paper, the
condition numbers are just averages of algebraic expressions
involving derivatives of the embedding $x_0$
and do not include any global assumption in the maps such as the
twist assumption.

Of course, KAM theory (and a fortiori KAM theory in
an a-posteriori format)  usually makes assumptions on
geometric properties of
 the dynamical system. Roughly, the geometric
properties are used to eliminate adding parameters to
the system.

There are different geometric properties
that lead to a KAM theory (see \cite{Moser67, BroerHS96}
for a  discusssion of the classical
contexts -- general, symplectic, volume preserving, reversible --
formulated in a format which is not a-posteriori).  Other more modern
contexts are presymplectic \cite{AlishahL12}, or closer to
the goals of this  paper, conformally symplectic \cite{CallejaCL11}.

Notice that an a-posteriori theorem allows to validate the
existence of an approximate solution, independently of
how it has been obtained. For example, one can take
as an approximate solution a numerically computed
one (typically this will be a trigonometric polynomial
whose coefficients are chosen among the numbers representable
in a computer).  If one can perform a finite (but  too large for pencil-paper)
number of operations taking care of the rounding off, one
can obtain estimates on $\varepsilon$ and $M_1,...,M_n$.

As it turns out, there exist computer science techniques
(interval analysis \cite{Moore79, Moore85, KaucherM84}) which allow one
to perform these rigorous bounds  mechanically. The coupling of an
a-posteriori theorem with interval arithmetic has led
to many \emph{computer assisted proofs} of mathematically
relevant problems that are reduced to the existence of a
fixed point\footnote{We note, however, that, besides computer
assisted proofs based on fixed point theorems,
there are other computer assisted proofs which do not involve
fixed points theorems, but which are based on
other arguments (exclusion of matches, algebraic operations, etc.).}.
 A particularly emblematic computer
assisted proof based on an a-posteriori theorem and interval arithmetic is
\cite{Lanford82}, but there are many other proofs based on
a-posteriori theorems for fixed points\footnote{ The proof of
\cite{Lanford82} used only a Banach contraction argument
and indeed most of the computer assisted proofs rely on
a contraction mapping argument. In  our case, we need to rely on
more sophisticated Nash-Moser arguments.}.

Therefore, a  way to prove the existence of
a quasi-periodic solution has different stages,
each of them requiring a different methodology.

\begin{itemize}
\item[A)] For a fixed geometric context, prove an a posteriori
KAM theorem.
\item[B)] Make sure that the conditions of the a-posteriori theorem  in  part A
are made explicit and computable.
\item[C)] Produce approximate solutions.
\item[D)] Verify the conditions given in B) on the approximate solutions
produced in C).
\end{itemize}

This strategy for two dimensional symplectic mappings was implemented
in \cite{Rana87} and in \cite{FHL}.  The paper \cite{Rana87} also
considered upper and lower bounds of Siegel radius and proved they would
converge to the right value if given enough computer resources.
The paper \cite{FHL} gives a very innovative implementation of a-posteriori
KAM estimates by proposing an efficient computer-assisted method. The technique
is successfully applied to the standard map, obtaining estimates in agreement
of $99.9\%$ with the numerical threshold. The paper has also considered applications
to the non-twist standard map and to the Froeschl\'e map.

We note that, in principle, the above methodology can continue in parameters
and establish the results even arbitrarily close to the values of the
parameters where the result is no longer true. Of course, in practice, one is
limited by the computer resources available (e.g. computer memory or time). We will show
that for some emblematic problems, even modest resources
(a common today's desktop) can produce results quite close
to optimal.

The parts A), B), C) and D) above require different methodologies
and are, in principle, independent. In practice, there are some
relations (e.g. the choice of spaces in the mathematical proofs
is related to the numerical methods used). This is why we decided
to present the results in a single paper rather than separate it
in logically independent units.

Part A) requires the traditional methods of mathematics, but the goals
should be an efficient and explicit formulation that makes efficient
the other parts of the strategy.
Notably, the functional equations
should involve functions of as little variables as possible -- the difficulty of dealing with functions grows very fast with the number of variables.
This is known as the \emph{curse of dimensionality}.
Moreover, the norms should be easy to evaluate. The spaces one
is dealing with should be easily parameterizable, preferably by linear
combinations of functions -- for example, parametrizing symplectic transformations requires using generating functions
to impose the very
nonlinear constraint of preserving the symplectic form.

Part B) is in principle straightforward, but a high quality implementation
requires taking advantage of the cancellations and organize the estimates
very efficiently.  Also, some non-constructive arguments need
to be replaced by constructive arguments.

Part C)  is very traditional in numerical analysis
and can be accomplished in many ways, for example discretizing the invariance
equation, but we stress that there are some interactions with the other parts.

Many of the more modern proofs in Part A)
are based in describing an iterative process and showing it
converges when started on a sufficiently approximate solution.
For our case, the proof presented in \cite{CallejaCL11} is
particularly well suited. It leads to a quadratically convergent
algorithm that requires little storage and a small operation count
per step.  On the other hand, the Newton method
relies on having a good approximate solution. The algorithm can
be used as the basis of a continuation method. Notice that the
method does not rely on indirect methods such as \cite{Greene79}
and that it is generally applicable
(i.e., one can take any system and let the continuation run).
The method of \cite{Greene79} requires continuing high period orbits,
which is problematic in systems with several harmonics (\cite{FalcoliniL92a,Lom-Call-06}).

We also note that, in order to have an effective part D), the discretization
used has to be such that it allows the evaluation of the norms involved.
As indicated above, the KAM theorem requires derivatives of rather high order,
so it seems that a Fourier discretization could be effective if
we consider norms that can be read from the Fourier coefficients.
This is particularly effective because the functional equations used in
part A) lead to functions whose maximal domain is a complex  strip (the domain
is invariant under an irrational translation), which are
the natural domains of convergence of Fouriers series.

Part D) is in principle straightforward since
the number of operations is rather small.  As mentioned before,
it can be made fully rigorous using interval arithmetic.

The goal of this paper is to  implement this strategy for
conformally symplectic mappings and obtain concrete results
for an emblematic example that has been considered many times
in the literature.  One caveat is that for part D), we have
not implemented interval arithmetic, but have performed
the calculations with more than 100 digits of precision
and with several precisions.

\subsection{Organization of the paper}
\label{sec:organization}

This paper is organized as follows.  In Section~\ref{sec:preliminaries},
we present some standard preliminaries, including norms, Cauchy estimates,
the Diophantine inequality, the solution of the cohomology equation,
the definition of conformally symplectic systems, the introduction and
properties of the dissipative standard map.

In Section~\ref{sec:mainthm}
we state  a very explicit KAM theorem in an a-posteriori format, Theorem~\ref{main}
which implements part A) of the strategy indicated above.
The statement of Theorem~\ref{main} includes the explicit formulation of the
smallness conditions on the parameters ensuring the existence of an exact solution
of the invariance equation. Such conditions depend on a set of constants, whose
explicit expression is given in Appendix~\ref{app:constants}).
The proof of Theorem~\ref{main} is reviewed in Section~\ref{sec:sketch}.
The proof follows closely the proof in \cite{CallejaCL11}, but we
take advantage that we will consider a specific model in which the tori are one--dimensional
and such that the symplectic form is the standard one.
Some of
the most straigtforward calculations have been relegated to Appendix~\ref{app:constants}.

\section{Preliminaries}
\label{sec:preliminaries}

In this Section, we collect several notions that play a role in our
results. We will describe the models, some of their properties
and describe the examples. The  material in Section~\ref{sec:lemmas}
concerns standard properties of
analytic functions and can be used mainly as a reference for the notation.
In Section~\ref{sec:conformallysymplectic}
we introduce conformally symplectic systems, which are
the main geometric assumption in our results. In Section~\ref{sec:model}
we introduce the concrete model we will study and which has been
widely investigated in the literature.

\subsection{Norms and preliminary Lemmas}\label{sec:lemmas}
In this Section we need to specify the norms
(see Section~\ref{sec:norms}), to estimate the composition of functions (see Section~\ref{sec:comp}), to bound
derivatives (see Section~\ref{sec:der}), to introduce Diophantine numbers (see Section~\ref{sec:Diophantine}),
and to give estimates of a cohomology equation associated to the
linearization of the invariance equation (see Section~\ref{sec:cohomology}).

\subsubsection{Norms}\label{sec:norms}
For a vector $v=\left( \begin{array}{c} v_1 \\ v_2 \\ \end{array} \right)\in\real^2$, we define its norm as
$$
||v||=|v_1|+|v_2|\ .
$$
For a matrix $A=\left( \begin{array}{cc} a_{11}&a_{12} \\ a_{21}&a_{22} \\ \end{array} \right)\in\real^2\times\real^2$,
we define its norm as
$$
||A||=\max\Big\{|a_{11}|+|a_{21}|,\ |a_{12}|+|a_{22}|\Big\}\ .
$$
To define the norm of functions and vector functions, we start by introducing
for $\rho>0$ the following complex extensions of a torus $\torus$, of a set $B$
and of the manifold $\M=B\times\torus$:
\beqano
\torus_\rho&\equiv& \{z=x+iy\in\complex/\integer:\ x\in\torus\ ,\ |y|\leq\rho\}\ ,\nonumber\\
B_\rho&\equiv &\{z=x+iy\in\complex:\ x\in B\ ,\quad |y|\leq \rho\}\ ,\nonumber\\
\M_\rho&=&B_\rho\times \torus_\rho\ .
\eeqano

We denote by $\A_{\rho}$ the set of functions which are analytic in $\Int (\torus_\rho)$ and that
extend continuously to the boundary of $\torus_\rho$. Within such set, we introduce the norm
$$
\|f\|_{\rho}=\sup_{z\in\torus_\rho} |f(z)|\ .
$$
For a vector valued function $f=(f_1,f_2,...,f_n)$, $n\geq 1$, we define the norm
\beq{normv}
\|f\|_{\rho}=\|f_1\|_{\rho}+\|f_2\|_{\rho}+...+\|f_n\|_{\rho}\ .
\eeq
For an $n_1\times n_2$ matrix valued function $F$ we define
\beq{normm}
\|F\|_{\rho}=\sum_{i=1}^{n_1} \sup_{j=1,...,n_2} \|F_{ij}\|_{\rho}\ .
\eeq
Notice that if $F$ is a matrix valued function and $f$ is a vector valued function, then one has
$$
\|F\,f\|_{\rho}\leq \|F\|_{\rho}\ \|f\|_{\rho}\ .
$$

\subsubsection{Composition Lemma}\label{sec:comp}
Composition of two functions is an important operation in
dynamical systems. Indeed, our main functional equation, see
\eqref{invariance} below, involves composition.

\begin{lemma}\label{lem:comp}
Let $F \in \A_{\C}$ be an analytic function on a
domain $\C\subset \complex\times \complex/\integer $.

Assume that the function $g$ is such that
$g(\torus_{\rho}) \subset \C$ and $g\in\A_{\rho}$ with $\rho >0$. Then, $F\circ g \in \A_{\rho}$
and
\[
\|F \circ g\|_{\rho} \leq \|F\|_{\A_{\C}}\ ,
\]
where $\|F\|_{\A_{\C}}= \sup_{z\in{\C}}\ |F(z)|$.

If, furthermore, we have that $\dist( g( \A_\rho), \complex \setminus \C) = \eta > 0$, then we have:
\begin{itemize}
\item
For all $h \in \A_\rho$ with $\| h \|_\rho < \eta/4 $, we can define
$F \circ( g + h)$.
\item
We have:
\[
\| F\circ(g + h) - F \circ g \|_\rho \le
\sup_{z, \dist(z, \C) \le \eta/4}( |DF(z)|)\ \| h\|_\rho\ ,
\]

\[
\| F\circ(g + h) - F \circ g - DF\circ g \, h \|_\rho \le
 \frac{1}{2}\sup_{z, \dist(z, \C) \le \eta/4}( |D^2F(z)|) \  \| h\|_\rho^2\ .
\]
\end{itemize}
\end{lemma}

\subsubsection{Cauchy estimates on the derivatives}\label{sec:der}
Estimates on the derivatives will be needed throughout the whole proof of the main result
(Theorem~\ref{main}).

\begin{lemma} \label{lem:Cauchy}
For a function $h\in\A_{\rho}$, we have the following estimate on the first derivative
on a smaller domain:
\beq{cc}
\|Dh\|_{\rho-\delta}\leq C_c\ \delta^{-1}\ \|h\|_\rho\ ,\qquad C_c=1\ ,
\eeq
where $0<\delta<\rho$.
For the $\ell$--th order derivatives with $\ell\geq 1$, one has:
$$
\|D^\ell h\|_{\rho-\delta}\leq C_{c,\ell}\ \delta^{-\ell}\ \|h\|_\rho\ ,\qquad C_{c,\ell}=\ell!\ (2 \pi)^{-1}\ .
$$
\end{lemma}

Notice that the Cauchy constant $C_c$ might assume different values,
if one adopts different norms with respect to \equ{normv}, \equ{normm} (this is why we keep a
symbol for such constant).

\subsubsection{Diophantine numbers}\label{sec:Diophantine}

The following definition is standard in number theory and appears
frequently in KAM theory.

\begin{definition}
Let $\omega\in\real$, $\tau\geq 1$, $\nu\geq 1$. We say that $\omega$ is Diophantine of
class $\tau$ and constant $\nu$, if the following inequality is satisfied:
\begin{equation}
|\omega \,k-q|\ \geq\ \nu |k|^{-\tau}\ ,\qquad q\in \zed\ ,\quad
k\in \zed\backslash\{0\}\ .
\label{DC}
\end{equation}
\end{definition}

The set of Diophantine numbers satisfying \equ{DC}
is denoted by $\D(\nu,\tau)$. The union over $\nu>0$ of the sets $\D(\nu,\tau)$ has full Lebesgue measure in $\real$.\\

\subsubsection{Estimates on the cohomology equation}\label{sec:cohomology}
Given any Lebesgue measurable function $\eta$, we consider the following cohomology equation:
\beq{difference}
\varphi(\theta+\omega)-\lambda \varphi(\theta)=\eta(\theta)\ ,\qquad \theta\in\torus\ .
\eeq
The solution of an equation of the form \equ{difference} will be an essential
ingredient of the proof, see e.g. \equ{ceq} below.
The two following Lemmas show that there is one Lebesgue measurable function $\varphi$, which is the solution
of \equ{difference}. Precisely, Lemma~\ref{contractive} applies for $|\lambda| \ne1$, $\omega\in \real$ and it
provides a non-uniform estimate on the solution, while Lemma~\ref{neutral} applies to any $\lambda$ and any
$\omega$ Diophantine, and it provides a uniform estimate on the solution.

\begin{lemma}\label{contractive}
Assume $|\lambda| \ne1$, $\omega\in \real$.
Then, given any Lebesgue  measurable function $\eta$, there is one
Lebesgue  measurable function $\varphi$
satisfying \eqref{difference}.
Furthermore, the following estimate holds:
$$
\|\varphi\|_{\rho}  \le\big|\,|\lambda| -1\big|^{-1}\|\eta\|_{\rho}\ .
$$
Moreover, one can bound the derivatives of $\varphi$ with respect to $\lambda$ as
$$
\|D_\lambda^j \varphi\|_{\rho}
\le {j!\over {\big|\,|\lambda|-1\big|^{j+1}}}\ \|\eta\|_{\rho}\ ,\qquad j\geq 1\ .
$$
\end{lemma}


\begin{lemma}\label{neutral}
Consider \eqref{difference} for $\lambda \in [A_0, A_0^{-1}]$ for some
$0 < A_0< 1$ and let $\omega\in \D(\nu,\tau)$. Assume that $\eta \in
\A_{\rho}$, $\rho>0$ and that
$$
\int_{\torus} \eta(\theta)\, d\theta =0\ .
$$
Then, there is one and  only one solution of \eqref{difference}
with zero average: $\int_{\torus} \varphi(\theta)\, d\theta =0$. Furthermore, if $\varphi \in \A_{\rho-\delta}$ for
$0<\delta<\rho$, then we have
\beq{estimate}
\|\varphi\|_{\rho-\delta}
\le C_0\ \nu^{-1}\ \delta^{-\tau} \|\eta\|_{\rho}\ ,
\eeq
where
\beq{Cu}
C_0={1\over {(2\pi)^\tau}}\ {\pi\over {2^\tau (1+\lambda)}}\ \sqrt{{\Gamma(2\tau+1)}\over 3}\ .
\eeq
\end{lemma}

The proof of Lemma~\ref{neutral} with the constant $C_0$ as in \equ{Cu} is
given in Appendix~\ref{app:lemmaproof}.

\subsection{Conformally symplectic systems}
\label{sec:conformallysymplectic}
In this Section we give the definition of conformally symplectic systems for one--dimensional maps.
Indeed, the dissipative standard map that we will introduce in Section~\ref{sec:model} and that we will consider
throughout this paper, is a one--dimensional, conformally symplectic map.
A more general definition of a conformally symplectic system in the $n$--dimensional case is provided in \cite{CallejaCL11}.\\

\begin{definition}\label{def:cs}
Let $\M$ be an analytic symplectic manifold with $\M\equiv B\times\torus$, where $B\subseteq\real$ is an
open, simply connected domain with a smooth boundary. Let $\Omega$ be the symplectic form associated to $\M$.
Let $f$ be a diffeomorphism defined on the phase space $\M$.
The diffeomorphism $f$ is conformally symplectic, if there exists a function $\lambda:\M\rightarrow \real$ such that
$$
f^*\Omega = \lambda\Omega\ ,
$$
where $f^*$ denotes the pull-back of $f$.
\end{definition}

We remark that
when $n=1$, then $\lambda$ can be a function of the coordinates, while it can be shown that for $n\geq 2$ one
can only have that $\lambda$ is a constant function.

In the following discussion, we will always assume that $\lambda$ is a constant, as in the model
\equ{dsm} below, which is the main goal of the present work.

\subsection{A specific model}
\label{sec:model}

In this work we consider a specific 1--parameter family $f_\mu$ of
one--dimensional, conformally symplectic maps, known as the \sl dissipative standard map: \rm
\beqa{dsm}
I'&=&\lambda I+\mu+{\ep \over 2 \pi}\sin(2 \pi \varphi)\ ,\nonumber\\
\varphi'&=&\varphi+I'\ ,
\eeqa
where $I\in B\subseteq\real$ with $B$ as in Definition~\ref{def:cs}, $\varphi\in\torus$, $\varepsilon\in\real_+$, $\lambda\in\real_+$, $\mu\in\real$.
This model has been studied both numerically and theoretically in the
literature. For example \cite{Rand92a, Rand92b, Rand92c} consider
the breakdown and conjecture universality properties;
\cite{CallejaC10} studies the breakdown even for complex values of the parameters;
\cite{CallejaF11} studies the invariant bundles near the circles
and find scaling properties at breakdown; \cite{BustamanteC19, CalCelLla17} study the domains
of analyticity in the limit of small dissipation.

To fix some terminology, we shall refer to $\varepsilon$ as the \sl perturbing parameter, \rm to $\lambda$ as
the \sl dissipative parameter, \rm and to $\mu$ as the \sl drift parameter. \rm

Notice that the Jacobian of the mapping \equ{dsm} is equal to $\lambda$, so that the mapping is contractive for $\lambda<1$,
expanding for $\lambda>1$ and it is symplectic for $\lambda=1$.

We denote by $(\cdot,\cdot)$ the Euclidean scalar product. We remark that if $J=J(x)$ is the matrix representing
$\Omega$ at $x$, namely $\Omega_x(u,v)=(u,J(x)v)$ for any $u$, $v\in\real$, then for the mapping \equ{dsm},
$J$ is the following constant matrix:
\beq{matrixJ}
J=\left(
\begin{array}{cc}
  0 & 1 \\
  -1 & 0 \\
\end{array}
\right)\ .
\eeq

\subsubsection{Formulation of the problem of an invariant attractor}
We proceed to provide the definition of a KAM attractor with frequency $\omega$.

Having fixed a value of the dissipative parameter,
our goal will be to prove the persistence of invariant attractors associated to \equ{dsm} for non-zero
values of the perturbing parameter.
To this end, we need to require that the frequency of the attractor,
say $\omega\in\real$, is Diophantine according to
Definition~\equ{DC}.  We note that this will require
adjusting  the drift parameter $\mu$.

\begin{definition}\label{def:inv}
Given a family of conformally symplectic maps $f_\mu:\M\rightarrow\M$,
a KAM attractor with frequency $\omega$ is an invariant torus which can be described
by an embedding $K:\torus\rightarrow\M$, such that the following invariance equation is satisfied for all $\theta\in\torus:$
\beq{invariance}
f_\mu\circ K(\theta)=K(\theta+\omega)\ .
\eeq
\end{definition}

The equation \eqref{invariance} will be the key of our statements.
Note that we will think that both the embedding $K$ and the parameter
$\mu$ are unknowns of \eqref{invariance}.

\begin{remark}\label{rem3}
$(i)$ For the dissipative standard map \equ{dsm} the embedding $K$ can be conveniently written as
\beq{uv}
K(\theta)=\left(
\begin{array}{c}
  \theta+u(\theta) \\
  v(\theta)\\
\end{array}
\right)
\eeq
for some continuous, periodic functions $u:\torus\rightarrow\real$, $v:\torus\rightarrow\real$.
Denoting by $(I_j,\varphi_j)$ the $j$-th iterate of \equ{dsm}, one finds that orbits are characterized by
$$
\varphi_{j+1}-(1+\lambda)\varphi_j+\lambda\, \varphi_{j-1}=\mu +
{\varepsilon \over 2 \pi}\sin(2 \pi \varphi_j)\ .
$$
Using \equ{uv} one obtains that the invariance equation \equ{invariance} in terms of $u$ is
\beq{u}
u(\theta+\omega)-(1+\lambda)u(\theta)+\lambda\, u(\theta-\omega)=\mu+
{\varepsilon \over 2 \pi} \sin(2 \pi (\theta+u(\theta)))\ .
\eeq
Equation \equ{u} can be used to determine the function $u$ and then one can
determine the function $v$ appearing in \equ{uv} by
$$
v(\theta)=\omega+u(\theta)-u(\theta-\omega)\ .
$$

$(ii)$ It is interesting to notice that for $\varepsilon=0$ the embedding can be chosen as $K(\theta)=(\theta,\omega)$.
In this case, the mapping \equ{dsm} admits a natural attractor with frequency $\omega=\mu/(1-\lambda)$.
This simple observation highlights the role of the drift $\mu$ and its relation to the
frequency $\omega$.
\end{remark}

The existence of an invariant attractor for $\varepsilon\not=0$ will be established by
fixing the frequency $\omega$ and determining
a solution $(K,\mu)$ (equivalently $(u,\mu)$ according to Remark~\ref{rem3}), satisfying the invariance
equation \equ{invariance} (equivalently \equ{u}) for a fixed value of the
dissipative parameter $\lambda$.
The focus of this paper will be in giving explicit estimates and showing
that the hypotheses of the theorem are satisfied numerically in a concrete example
for explicit values of $\varepsilon$, $\lambda$.
In particular, we will verify numerically that the
estimates of the theorem are satisfied taking a numerically
computed solution as the approximate solution. The computation of the solution is described in
Section~\ref{sec:approximate}.
The  verification of the estimates on
these numerical solutions is presented in Section~\ref{sec:value}.

\section{A KAM Theorem}\label{sec:mainthm}
In this Section we state the main mathematical
result, Theorem~\ref{main}, which is a KAM result in the a-posteriori format described
in Theorem Format~\ref{aposteriori}. Theorem~\ref{main} specifies
some condition numbers to be measured in the approximate solution.  It
shows that, if there is a function $K_0$ and a number $\mu_0$ that,
when substituted in  \eqref{invariance}, give a residual
(measured in a norm that we specify) which is smaller than
a function of the condition numbers, then, there is a solution of
\eqref{invariance} close (in some norm that we specify) to $K_0, \mu_0$.

We also note that the method of proof, which is based on constructing an
iterative procedure, leads to a very efficient algorithm.  Later,
we will describe the implementation of the algorithm and the verification of
the estimates required in Theorem~\ref{main}.

For an embedding $K_0=K_0(\theta)$ and a frequency $\omega$, we start by introducing some auxiliary quantities defined as follows:
\beqa{def}
M_0(\theta) &\equiv& [ DK_0(\theta)\ |\  J^{-1}\circ K_0(\theta)\ DK_0(\theta) N_0(\theta)]\ ,\nonumber\\
S_0(\theta) & \equiv& ((DK_0 N_0)\circ T_\omega)^\top(\theta) Df_{\mu_0} \circ K_0(\theta) J^{-1}\circ K_0(\theta)DK_0(\theta)N_0(\theta)\ ,\nonumber\\
N_0(\theta) &\equiv& (DK_0(\theta)^\top DK_0(\theta))^{-1}\ ,
\eeqa
where the superscript $\top$ denotes the transposition
and $T_\omega$ denotes the shift by $\omega$: for a function $P=P(\theta)$, then $(P\circ T_\omega)(\theta)=P(\theta+\omega)$.

The following Theorem provides a constructive version of Theorem 20 in \cite{CallejaCL11} for
mappings as those introduced in Definition~\ref{def:cs}; in particular, it applies to the
dissipative standard map \equ{dsm}.

It is quite important that the condition numbers in Theorem~\ref{main} are properties of
the approximate solution, not global properties of the map. The condition numbers can be
computed from the approximate solution by taking derivatives, performing algebraic operations and averaging.

\begin{theorem}\label{main}
Consider a family $f_\mu:\M\rightarrow\M$ of conformally symplectic mappings, defined on the manifold $\M\equiv B\times \torus$ with
$B\subseteq\real$ an open, simply connected domain with a smooth boundary.
Let the mappings $f_\mu$ be analytic on an open connected
domain $\mathcal{C}\subset\complex\times\complex/\integer$.
Let the following assumptions be satisfied.

{\bf H1} Let $\omega\in \D(\nu,\tau)$ as in \equ{DC}.

{\bf H2} There exists an approximate solution $(K_0,\mu_0)$ with $K_0\in\mathcal{A}_{\rho_0}$ for
some $\rho_0>0$ and with $\mu_0\in\Lambda$, $\Lambda\subset\real$ open. Let $(K_0,\mu_0)$ be
such that \eqref{invariance} is satisfied up to an error function $E_0=E_0(\theta)$, namely
$$
f_{\mu_0}\circ K_0(\theta)-K_0(\theta+\omega)=E_0(\theta)\ .
$$
Let $\varepsilon_0$ denote the size of the error function, i.e.
$$
\varepsilon_0\equiv \|E_0\|_{\rho_0}\ .
$$

{\bf H3} Assume that the following non--degeneracy condition holds:
$$
\det
\left(
\begin{array}{cc}
  {\overline S}_0 & {\overline {S_0(B_{b0})^0}}+\overline{\widetilde A_0^{(1)}} \\
  \lambda-1 & \overline{\widetilde A_0^{(2)}} \\
 \end{array}%
\right) \ne 0\ ,
$$
where $S_0$ is given in \eqref{def},
$\widetilde A_0^{(1)}$, $\widetilde A_0^{(2)}$ denote the first and second elements of the vector
$\widetilde A_0\equiv M_0^{-1}\circ T_{\omega} D_\mu f_{\mu_0} \circ K_0$,
$(B_{b0})^0$ is the solution (with zero average in the $\lambda = 1$ case) of the equation
$\lambda (B_{b0})^0-(B_{b0})^0\circ T_\omega=-(\widetilde A_0^{(2)})^0$, where $(\widetilde A_0^{(2)})^0$
denotes the zero average part of $\widetilde A_0^{(2)}$.
Denote by ${\mathcal T}_0$ the \emph{twist constant} defined as
$${\mathcal T}_0  \equiv \left \|\left(%
\begin{array}{cc}
  {\overline S}_0 & {\overline {S_0(B_{b0})^0}}+\overline{\widetilde A_0^{(1)}} \\
  \lambda-1 & \overline{\widetilde A_0^{(2)}} \\
 \end{array}%
\right)^{-1} \right \|\ .
$$

{\bf H4} Assume there exists $\zeta>0$, so that
$$
{\rm dist}( \mu_0, \partial \Lambda) \ge  \zeta\ ,\qquad
{\rm dist}(K_0(\torus_{\rho_0}), \partial\mathcal{C})  \ge \zeta\ .
$$

{\bf H5} Let $0<\delta_0<\rho_0$.
Let $\kappa_\mu\equiv 4 C_{\sigma 0}$ with
$C_{\sigma 0}$ constant (whose explicit expression is given in Appendix~\ref{app:constants}).
Let the quantities $Q_0$, $Q_{\mu 0}$, $Q_{z\mu 0}$, $Q_{\mu\mu 0}$, $Q_{E 0}$ be defined as
\beqa{QQQ}
Q_0&\equiv&\sup_{z\in\mathcal{C}}|Df_{\mu_0}(z)|\ ,\nonumber\\
Q_{\mu 0}&\equiv&\sup_{z\in\mathcal{C}}|D_\mu f_{\mu_0}(z)|\ ,\nonumber\\
Q_{z\mu 0}&\equiv&\sup_{z\in\mathcal{C},\mu\in\Lambda,|\mu-\mu_0|<2\kappa_\mu\varepsilon_0}|D_\mu Df_{\mu}(z)|\ ,\nonumber\\
Q_{\mu\mu 0}&\equiv&\sup_{z\in\mathcal{C},\mu\in\Lambda,|\mu-\mu_0|<2\kappa_\mu\varepsilon_0}|D^2_\mu f_{\mu}(z)|\ ,\nonumber\\
Q_{E 0}&\equiv&{1\over 2}\max\Big\{\|D^2E_0\|_{{\rho_0-\delta_0}},\|DD_\mu E_0\|_{{\rho_0-\delta_0}},
\|D^2_\mu E_0\|_{{\rho_0-\delta_0}}\Big\}\ .
\eeqa
Assume that $\varepsilon_0$ satisfies the following smallness conditions for suitable
real constants $C_{\eta 0}$, $C_{\E 0}$, $C_{d0}$, $C_{\sigma 0}$, $C_\sigma$,
$C_{W0}$, $C_W$, $C_{\mathcal{R}}$
(see Appendix~\ref{app:constants} for their explicit expressions):
\beq{C1}
C_{\eta 0}\,\nu^{-1}\delta_0^{-\tau}\varepsilon_0<\zeta\ ,
\eeq
\beq{C2}
2^{3\tau+4}\,C_{\E 0}\ \nu^{-2}\ \delta_0^{-2\tau} \varepsilon_0\leq 1\ ,
\eeq
\beq{C3}
4C_{d0} \nu^{-1}\delta_0^{-\tau} \varepsilon_0<\zeta\ ,
\eeq
\beq{C4}
4C_{\sigma 0}\varepsilon_0<\zeta\ ,
\eeq
\beq{condbT}
\|N_0\|_{\rho_0}\ (2\|DK_0\|_{\rho_0}+D_K)\ D_K<1
\eeq
\beq{Cnew1}
4 Q_{z\mu 0} C_{\sigma 0} \varepsilon_0< Q_0\ ,
\eeq
\beq{Cnew2}
4 Q_{\mu\mu 0} C_{\sigma 0} \varepsilon_0< Q_{\mu 0}\ ,
\eeq
\beq{C8}
C_\sigma\ D_K\leq C_{\sigma 0}\ ,
\eeq
\beq{C9}
D_K(C_{W0}+\|M_0\|_{\rho_0}C_W+C_W D_K)\leq C_{d0}\ ,
\eeq
\beq{C10}
D_K\ \Big(C_W\ C_c\ \nu \delta_0^{-1+\tau}+C_{\mathcal{R}}\Big)
  \leq C_{\E 0}\ ,
\eeq
where $D_K$ is defined as
\beq{DK1}
D_K\equiv 4C_{d0}\ C_c\ \nu^{-1}\delta_0^{-\tau-1}\ \varepsilon_0
\eeq
and with $C_c$ as in \equ{cc}.

\vskip.1in

Then, there exists an exact solution $(K_e,\mu_e)$ of \equ{invariance} such that
$$
f_{\mu_e} \circ K_e - K_e \circ T_\omega = 0\ .
$$
The quantities $K_e$, $\mu_e$ are close to the approximate solution, since one has
\beqa{kemu}
|| K_e - K_0 ||_{\rho_0-\delta_0} &\le& 4 C_{d0}\, \nu^{-1}\, \delta_0^{-\tau}\,  ||E_0||_{\rho_0}\ ,\nonumber\\
| \mu_e - \mu_0| &\le& 4 C_{\sigma 0}\, ||E_0||_{\rho_0}\ .
\eeqa
\end{theorem}

The explicit expressions of the constants entering in the conditions \equ{C1}-\equ{C10}
are obtained by implementing constructively the KAM proof presented in \cite{CallejaCL11} and
sketched in Section~\ref{sec:sketch}. In Section~\ref{sec:results} the family $f_\mu$ will be taken as
the dissipative standard map defined in \equ{dsm}; then, the explicit expressions for the constants
- provided in Appendix~\ref{app:constants} - will allow us to compute concrete values for $\varepsilon_0$,
once we fix the frequency $\omega$ and the conformal factor $\lambda$. Therefore,
the conditions \equ{C1}-\equ{C10} will allow us to obtain a lower bound on the perturbing parameter,
ensuring the existence of an invariant attractor with fixed frequency $\omega$ and for a given conformal factor $\lambda$.

\begin{remark}
  For any value of $\lambda$ with $|\lambda|<1$, Theorem~\ref{main} also
  ensures that the quasi--periodic solution provided
  by the manifold $K_e(\torus^n)$ is a local attractor and that the dynamics
  on this attractor is analytically conjugated to a rigid rotation.
Indeed, according to \cite{CallejaCL11b}, the diagrams in a neighborhood is analytically conjugated to
a rotation and homothety.
\end{remark}

\begin{remark}
One question that has been posed to us several times is how it is possible
to use the computer to verify hypotheses that involve irrational
numbers and indeed the Diophantine properties. After all,
the standard computer numbers are only rational numbers.

The answer is that the a-posteriori theorem uses
the Diophantine properties and that this theorem is indeed given a traditional
proof. To verify the hypothesis, we compute numerically
$\| f_\mu \circ K - K \circ T_{\omega_0}\| $ where $\omega_0$ is
indeed a rational number.

It is clear that for $\xi\in(\omega_0,\omega)$:
\[
\begin{split}
\| f_\mu \circ K - K \circ T_\omega\|
& \le \| f_\mu \circ K - K \circ T_{\omega_0}\| +
\| K \circ T_{\omega_0}   - K \circ T_\omega \| \\
&\le
\| f_\mu \circ K - K \circ T_{\omega_0}\| +
\|D K \circ T_\xi \| | \omega - \omega_0|\ .
\end{split}
\]
In our case, we see that $|\omega - \omega_0| \le 10^{-100}$ and that $DK$ is
a number of order $1$. Hence, the last term does not
affect the final result much.

Of course, implementing interval arithmetic, one can also
use an interval that contains the desired frequency and
obtain estimates for the error of invariance valid uniformly
for all $\omega$ in this interval.
\end{remark}

\section{Sketch of the proof of Theorem~\ref{main}}\label{sec:sketch}

We note that in the statement of Theorem~\ref{main} (and in the subsequent text) all the
constants are given explicitely (see appendix~\ref{app:constants}). There are only a few dozen of conditions
to check; all these conditions are easy algebraic expressions. Even if this is cumbersome for a human,
a computer calculates them in a very short time.

The proof of Theorem~\ref{main} is presented in detail in \cite{CallejaCL11}.
However, in \cite{CallejaCL11} the proof was given for a general case and no explicit estimates
on the constants were provided. On the contrary, in Section~\ref{sec:proof} we give concrete expressions in view
of the application to \equ{dsm}.

Before providing the lengthy and detailed proof given in Section~\ref{sec:proof}, we proceed to outline a
sketch of the proof of Theorem~\ref{main} by splitting it in 5 main steps, which are
needed to be performed in order to get the solution of the invariance equation.
The constructive version of the proof, which will be developed in Section~\ref{sec:proof}, yields explicit expressions for
the constants appearing in the smallness conditions \equ{C1}-\equ{C10};
being a long list, such constants are given in Appendix~\ref{app:constants}.

\vskip.1in

We anticipate that it is easy to see that in the one--dimensional case of the mapping \equ{dsm} all invariant
curves are Lagrangian; this observation will simplify the proof presented in Section~\ref{sec:proof}
with respect to that developed in \cite{CallejaCL11}.
However, in the general $n$-dimensional case the following remark gives a practical characterization of the Lagrangian
character of invariant tori.
\begin{remark}
Let $f$ be a conformally symplectic map defined on an $n$-dimensional manifold $B_n\times\torus^n$, where $B_n\subseteq\real^n$
is an open, simply connected domain with smooth boundary. Then,
invariant tori are \rm Lagrangian, \sl
namely if $|\lambda| \ne1$ and $K$ satisfies \eqref{invariance},
then one has
\begin{equation}
K^* \Omega  =0\ .
\label{lagrangian}
\end{equation}
Moreover, if $f$ is symplectic and $\omega$ is irrational, then the
$n$-dimensional torus is Lagrangian.
\end{remark}

Below it is a description of the main steps required to prove Theorem~\ref{main}.

\vskip.2in

\bf Step 1: \rm on the initial approximate solution and its linearization.

Let $(K_0,\mu_0)$ be an approximate solution of the invariance equation \equ{invariance}
and let $E_0=E_0(\theta)$ be the associated error function.
In coordinates, the Lagrangian condition \equ{lagrangian} becomes
$$
DK_0^\top(\theta)\ J\circ K_0(\theta)\ DK_0(\theta)=0\ ,
$$
which shows that the tangent space can be decomposed as
$\Range \Big(DK_0(\theta)\Big) \oplus \Range \Big(J^{-1}\circ K_0(\theta) DK_0(\theta)\Big)$.

Therefore, one can show that, up to a remainder function $R_0=R_0(\theta)$,
the quantity $Df_{\mu_0}\circ K_0$ is conjugate to an upper diagonal matrix with constant
diagonals and the following identity is satisfied:
\beq{R}
Df_{\mu_0} \circ K_0(\theta)\ M_0(\theta) = M_0(\theta +\omega)
\begin{pmatrix}\Id & S_0(\theta)\\ 0&\lambda \Id \end{pmatrix} +R_0(\theta)
\eeq
with $M_0$ and $S_0$ as in \equ{def}.


Then, we proceed to find some corrections $W_0$ and $\sigma_0$ such that, setting $K_1=K_0+M_0 W_0$, $\mu_1=\mu_0+\sigma_0$, one has that
the new approximation $(K_1,\mu_1)$ satisfies the following invariance equation:
\beq{E1}
f_{\mu_1}\circ K_1(\theta)-K_1(\theta+\omega) =E_1(\theta)
\eeq
for some error function $E_1=E_1(\theta)$. The requirement on $E_1$ is that its norm is
quadratically smaller than the norm of
the initial approximation $E_0$. This can be obtained provided that the
following equation is satisfied:
\beq{quad}
Df_{\mu_0} \circ K_0(\theta)\ M_0(\theta)W_0(\theta) - M_0(\theta+\omega)\ W_0(\theta+\omega) + D_\mu  f_{\mu_0} \circ
K_0(\theta)\sigma_0 = - E_0(\theta)\ .
\eeq

\vskip.2in

\bf Step 2: \rm determination of the new approximation.

The corrections $(W_0,\sigma_0)$ in Step 1 are determined as follows.
Using \equ{R}, \equ{quad} and neglecting higher order terms, one obtains two cohomology equations
with constant coefficients for $W_0$ and $\sigma_0$. More precisely, writing $W_0$ in components as $W_0=(W_0^{(1)},W_0^{(2)})$,
such cohomological equations are given by
\beqa{ceq}
W_0^{(1)}(\theta)-W_0^{(1)}(\theta+\omega) &=& -\widetilde E_0^{(1)}(\theta)-S_0(\theta) W_0^{(2)}(\theta)
-\widetilde A_0^{(1)}(\theta)\, \sigma_0\ ,\nonumber\\
\lambda W_0^{(2)}(\theta)-W_0^{(2)}(\theta+\omega) &=& -\widetilde E_0^{(2)}(\theta)-\widetilde A_0^{(2)}(\theta)\,\sigma_0
\eeqa
with $S_0$ given in \equ{def}, while $\widetilde E_0$, $\widetilde A_0$ are defined as
\beqa{EA}
\widetilde E_0&\equiv&(\widetilde E_0^{(1)},\widetilde E_0^{(2)})\equiv M_0^{-1}\circ T_\omega E_0\ ,\nonumber\\
\widetilde A_0&\equiv&M_0^{-1}\circ T_\omega\ D_\mu f_{\mu_0}\circ K_0\ ,
\eeqa
where we denote by $\widetilde A_0^{(1)}$, $\widetilde A_0^{(2)}$ the first and second elements of the vector
$\widetilde A_0$.

We remark that the first equation in \equ{ceq} involves small divisors. In fact, the
Fourier expansion of the l.h.s. of the first equation in \equ{ceq} is given by
$$
W_0^{(1)}(\theta)-W_0^{(1)}(\theta+\omega) =\sum_{k\in\integer} \widehat W_{0,k}^{(1)}\,
e^{2\pi ik\theta}(1-e^{2\pi ik \omega})\ .
$$
Then, we notice that for $k=0$ there appears the zero factor $1-e^{2\pi ik\omega}=0$.
On the other hand, the second equation in \equ{ceq} is always solvable for any $|\lambda|\not=1$ by a contraction mapping argument.

Let us split $W_0^{(2)}$ as $W_0^{(2)}=\overline{W}_0^{(2)}+(W_0^{(2)})^0$, where the first
term denotes the average of $W_0^{(2)}$ and the second term the zero--average part.
We remark that the average of $W_0^{(1)}$ can be set to zero without loss of generality.
On the other hand, computing the averages of the cohomological equations \equ{ceq}, one can determine
$\overline{W}_0^{(2)}$, $\sigma_0$ by solving the system of equations
\beq{average}
\left(
\begin{array}{cc}
  \overline{S_0} & \overline{S_0(B_{b0})^0} +\overline{\widetilde A_0^{(1)}}  \\
  \lambda-1 & \overline{\widetilde A_0^{(2)}}  \\
 \end{array}%
\right)
\left(
\begin{array}{cc}
\overline{W}_0^{(2)}  \\
\sigma_0\\
 \end{array}%
\right)=
\left(
\begin{array}{cc}
-\overline{S_0 (B_{a0})^0}- \overline{\widetilde E_0^{(1)}} \\
-\overline{\widetilde E_0^{(2)}} \\
 \end{array}%
\right)\ ,
\eeq
where we have split $(W_0^{(2)})^0$ as $(W_0^{(2)})^0=(B_{a0})^0+\sigma_0 (B_{b0})^0$, where $(B_{a0})^0$, $(B_{b0})^0$
are the zero average solutions of
\beqa{BB}
\lambda(B_{a0})^0-(B_{a0})^0\circ T_\omega&=&-(\widetilde E_0^{(2)})^0\ ,\nonumber\\
\lambda(B_{b0})^0-(B_{b0})^0\circ T_\omega&=&-(\widetilde A_0^{(2)})^0
\eeqa
with $(\widetilde E_0^{(2)})^0$, $(\widetilde A_0^{(2)})^0$ denoting the zero average parts of $\widetilde E_0^{(2)}$, $\widetilde A_0^{(2)}$.
After solving \equ{average}, one can proceed to solve \equ{ceq} for the zero average parts of $W_0^{(1)}$, $W_0^{(2)}$.


\vskip.2in

\bf Step 3: \rm on the quadratic convergence of the iterative step.

Once we have determined the correction $(W_0,\sigma_0)$ as in step 2, we show that the new solution $K_1=K_0+M_0 W_0$, $\mu_1=\mu_0+\sigma_0$
satisfies the invariance equation with an error quadratically smaller with respect to the error at the previous step.
Precisely, for $0<\delta_0<\rho_0$ one
can prove that the error term $E_1$ in \equ{E1} associated to $(K_1,\mu_1)$ satisfies an inequality of the form
$$
\|E_1\|_{\rho_0-\delta_0}\leq C'\nu^{-1}\delta_0^{-\tau}\|E_0\|_{\rho_0}^2\ ,
$$
for some constant $C'>0$ (of course, in this constructive version of the Theorem we will assume
explicit values of the new error).


\vskip.2in

\bf Step 4: \rm on the analytic convergence of the sequence of approximate solutions.

The procedure outlined in steps 1--3 can be iterated to get a sequence of approximate solutions, say $\{K_j,\mu_j\}$
with $j\geq 0$.
The convergence of the sequence of approximate solutions to the true solution of the invariance
equation \equ{invariance} is obtained through
an analytic smoothing, which provides the convergence of the
iterative step to the exact solution. It is worth noticing that the sequence of approximate
solutions is constructed in domains that are smaller than the original domains.
However, one can suitably choose the loss of analyticity domains, so that
the exact solution is defined in a non-empty domain.


\vskip.1in

\bf Step 5: \rm on the local uniqueness of the solution.

According to \cite{CallejaCL11}, if there exist two solutions $(K_a,\mu_a)$, $(K_b,\mu_b)$ close enough,
then there exists $s\in{\mathbb R}$ such that for all $\theta\in\torus$:
\beqano
K_b(\theta)&=&K_a(\theta+s)\ ,\nonumber\\
\mu_a&=&\mu_b\ .
\eeqano
We refer to \cite{CallejaCL11} for the proof of the uniqueness of the solution.

\section{A constructive version of the proof of Theorem~\ref{main}}\label{sec:proof}
In this Section we prove Theorem~\ref{main}, providing explicit
expressions for all the constants involved. Such quantities will depend on the norm of the mapping,
the initial parameters and the norm of the initial approximation.

According to the sketch of the proof given in Section~\ref{sec:sketch}, we proceed to compute
the estimates of the following quantities:

\begin{itemize}
\item[$(i)$] estimate of the remainder $R_0=R_0(\theta)$ appearing in \equ{R} (see Section~\ref{sec:lagr}
and compare with step 1 of Section~\ref{sec:sketch});

\item[$(ii)$] estimates for the corrections $(W_0,\sigma_0)$ defined as the solutions of the cohomological
equations \equ{ceq} (see Section~\ref{sec:increments} and compare with step 2 of Section~\ref{sec:sketch});

\item[$(iii)$] quadratic estimates on the convergence of the iterative step, which implies to bound
the norm of $D\E\ M_0 W_0+D_\mu\E\ \sigma_0+E_0$, where $\E$ is the error functional
$\E[K,\mu]\equiv f_\mu\circ K-K\circ T_\omega$ (see Section~\ref{sec:iterative});

\item[$(iv)$] thanks to the results in $(iii)$, we will be able to give quadratic estimates of the error associated to
the new solution, say $\E[K+\Delta,\mu+\sigma]$, in terms of the square of the norm of $E_0$
(see Section~\ref{sec:error} and compare with \equ{E1} and step 3 of Section~\ref{sec:sketch});

\item[$(v)$] proof of the analytic convergence of the sequence of approximate solutions to the true solution of the
invariance equation (see Section~\ref{sec:analytic} and compare with step 4 of Section~\ref{sec:sketch}).
\end{itemize}

\subsection{Estimate on the error $R_0$}\label{sec:lagr}
Any torus associated to a one--dimensional map is always Lagrangian; this leads to a
simplification of the expression of the remainder function in \equ{R}. In fact, recalling
the definition of $M_0$ in \equ{def}, it turns out that
\beq{VEAR}
R_0(\theta)=DE_0(\theta)\ .
\eeq

Using Cauchy estimates, a bound on $R_0$ in \equ{VEAR} is given by
\beq{R0}
\|R_0\|_{\rho_0 - \delta_0}\leq C_c\ \delta_0^{-1}\ \|E_0\|_{\rho_0}
\eeq
with $C_c$ as in \equ{cc}.

\subsection{Estimates for the increment in the steps}\label{sec:increments}
We proceed to give estimates for the corrections $W_0$ and
$\sigma_0$, which satisfy the equations \equ{ceq}.

\begin{lemma}\label{lem:stepestimates}
Let $K_0\in \A_{\rho_0+\delta_0}$, $K_0(\torus_{\rho_0}) \subset \text{\rm
domain }(f_\mu)$, $\dist(K_0(\torus_{\rho_0}), \partial(\text{\rm
domain}(f_\mu))) \geq \zeta>0$ with $\rho_0$, $\delta_0$, $\zeta$ as in Theorem~\ref{main}.
For any $|\lambda|\not=1$ we have
\beqano
\|W_0\|_{{\rho_0-\delta_0}} &\le&C_{W0}\, \nu^{-1}\, \delta_0^{-\tau}\|E_0\|_{\rho_0}\ ,\nonumber\\
|\sigma_0| &\le& C_{\sigma 0}\, \|E_0\|_{\rho_0}\ ,
\eeqano
where
\beqa{NU}
C_{\sigma 0}&\equiv&\mathcal{T}_0\ \Big[|\lambda-1|\
\Big({1\over {||\lambda|-1|}}\|S_0\|_{\rho_0}+1\Big)+\|S_0\|_{\rho_0}\Big]\|M_0^{-1}\|_{\rho_0}\ ,\nonumber\\
\overline{C}_{W_2 0} &\equiv& 2\, \mathcal{T}_0\
\Big({1\over {||\lambda|-1|}}\|S_0\|_{\rho_0}+1\Big)\ Q_{\mu 0}\ \|M_0^{-1}\|^2_{\rho_0}\ ,\nonumber\\
C_{W_2 0}&\equiv& {1\over {||\lambda|-1|}}\Big(1+C_{\sigma 0}\ Q_{\mu 0}\Big)\|M_0^{-1}\|_{\rho_0}\ ,\nonumber\\
C_{W_1 0}&\equiv&C_0\Big[\|S_0\|_{\rho_0} (C_{W_2 0}+\overline{C}_{W_2 0})+\|M_0^{-1}\|_{\rho_0}+Q_{\mu 0} \|M_0^{-1}\|_{\rho_0}C_{\sigma 0}\Big]\ ,\nonumber\\
C_{W0}&\equiv& C_{W_1 0}+(C_{W_2 0}+\overline{C}_{W_2 0})\nu\delta_0^{\tau}\ .
\eeqa
\end{lemma}

\begin{proof}
Let $Q_{\mu 0}$ be an upper bound on the norm of $D_\mu f_{\mu_0}$ as in \equ{QQQ}.
Let $\widetilde A_0$ be defined as in \equ{EA}; then, we have:
$$
\|\widetilde A_0\|_{\rho_0} \leq Q_{\mu 0} \|M_0^{-1}\|_{\rho_0}\ .
$$
Recalling the definition of $S_0$ in \equ{def}, we obtain
$$
\|S_0\|_{{\rho_0}} \leq J_e \ Q_0\ \|DK_0\|_{\rho_0}^2 \|N_0\|_{\rho_0}^2\leq
C_c^2 J_e\ \ Q_0\ \|K_0\|_{{\rho_0+\delta_0}}^2 \|N_0\|_{\rho_0}^2\ \delta_0^{-2}\ ,
$$
where we used the estimate $\|DK_0\|_{\rho_0}\leq C_c \|K_0\|_{{\rho_0+\delta_0}}\ \delta_0^{-1}$ and where
$J_e$ denotes the norm of the symplectic matrix $J$ in \equ{matrixJ} (the norm of
$J^{-1}$ is again bounded by $J_e$). Notice that, with the choice of the norms
in Section~\ref{sec:norms} it is $J_e=1$.
We notice that, recalling the definition of $S_0$ and $M_0$ in \equ{def},
one can compute directly the functions and evaluate their norm.

For any $|\lambda|\not=1$, we have the
estimates given below, which follow from \equ{EA}, \equ{average}, \equ{BB}:
\beqano |{\overline W_0^{(2)}}| &\leq& {\mathcal T_0}\Big(\|{\overline{S_0(B_{b0})^0}}+
\overline{\widetilde A_0^{(1)}}\|_{\rho_0}\,
\|\overline{\widetilde E_0^{(2)}}\|_{\rho_0}+
\|{\overline{S_0(B_{a0})^0}}+\overline{\widetilde E_0^{(1)}}\|_{\rho_0}\,
\|\overline{\widetilde A_0^{(2)}}\|_{\rho_0}\Big)\nonumber\\
&\leq&{\mathcal T}_0\Big[\Big({1\over {||\lambda|-1|}}\|S_0\|_{\rho_0}+1\Big)\,
Q_{\mu 0}\|M_0^{-1}\|_{\rho_0}^2\, \|E_0\|_{\rho_0}\nonumber\\
&+&\Big({1\over {||\lambda|-1|}}\|S_0\|_{\rho_0}+1\Big)\, \|M_0^{-1}\|_{\rho_0}^2\
\|E_0\|_{\rho_0} Q_{\mu 0}\Big]\nonumber\\
&\leq&2{\mathcal T}_0\Big({1\over {||\lambda|-1|}}\,
\|S_0\|_{\rho_0}+1\Big)\, Q_{\mu 0}\ \|M_0^{-1}\|_{\rho_0}^2\|E_0\|_{\rho_0}\nonumber\\
&\equiv&\overline{C}_{W_2 0}\ \|E_0\|_{\rho_0}\ ,\nonumber\\
|\sigma_0| &\leq& {\mathcal T}_0\Big(|\lambda-1|\,\|{\overline{S_0(B_{a0})^0}}+\overline{\widetilde E_0^{(1)}}\|_{\rho_0}+\|S_0\|_{\rho_0}\,
\|\overline{\widetilde E_0^{(2)}}\|_{\rho_0}\Big)\nonumber\\
&\leq&{\mathcal T}_0\Big[|\lambda-1|\ \Big({1\over {||\lambda|-1|}}\|S_0\|_{\rho_0}+1\Big)+
\|S_0\|_{\rho_0}\Big]\|M_0^{-1}\|_{\rho_0}\|E_0\|_{\rho_0}\nonumber\\
&\equiv&C_{\sigma 0}\ \|E_0\|_{\rho_0}
\eeqano
with $\overline{C}_{W_2 0}$, $C_{\sigma 0}$ as in \equ{NU}. Then, using
Lemma~\ref{contractive} and Lemma~\ref{neutral}, we have:
\begin{equation}
\label{Wnonuniformanalytic}
\begin{split}
    \|(W_0^{(2)})^0\|_{\rho_0} &\leq  \frac{1}{\big| |\lambda| -1\big|}
    \left(\|M_0^{-1}\|_{\rho_0} \|E_0\|_{\rho_0}+Q_{\mu 0}\|M_0^{-1}\|_{\rho_0}\ |\sigma_0|\right)\leq C_{W_2 0}\ \|E_0\|_{\rho_0}\ ,\nonumber\\
    \|W_0^{(1)}\|_{\rho_0-\delta_0} &\leq
    C_0 \nu^{-1} \delta_0^{-\tau} \Big( \|S_0\|_{\rho_0}\|W_0^{(2)}\|_{\rho_0}+
    \|M_0^{-1}\|_{\rho_0}\|E_0\|_{\rho_0}+Q_{\mu 0}\|M_0^{-1}\|_{\rho_0}|\sigma_0|\Big)\nonumber\\
    &\leq C_{W_1 0}\ \nu^{-1}\delta_0^{-\tau}\|E_0\|_{\rho_0}
\end{split}
\end{equation}
with $C_{W_1 0}$, $C_{W_2 0}$ as in \equ{NU}. In conclusion, we obtain:
\beqano
\|W_0\|_{\rho_0-\delta_0} &\leq&\|W_0^{(1)}\|_{\rho_0-\delta_0}+\|W_0^{(2)}\|_{\rho_0}\leq
C_{W_1 0}\nu^{-1}\delta_0^{-\tau}\|E_0\|_{\rho_0}+(C_{W_2 0}+\overline{C}_{W_2 0})\|E_0\|_{\rho_0}\nonumber\\
&\equiv&C_{W0}\ \nu^{-1}\delta_0^{-\tau}\|E_0\|_{\rho_0}\ ,
\eeqano
with $C_{W0}$ as in \equ{NU}.
\end{proof}

\begin{remark}
Let us define the error functional
$$
  \E[K_0,\mu_0] \equiv f_{\mu_0} \circ K_0 - K_0 \circ T_\omega\ .
$$
Let
$$
(\Delta_0, \sigma_0) = -\eta[K_0, \mu_0] E_0\ ,
$$
where $\Delta_0=-(\eta[K_0, \mu_0] E_0)_1$, $\sigma_0=-(\eta[K_0, \mu_0] E_0)_2$.
Then, using that $\Delta_0=M_0W_0$, one has:
\beqano
\|\eta[K_0, \mu_0] E_0\|_{\rho_0-\delta_0}&\leq&\|M_0\|_{\rho_0}\|W_0\|_{\rho_0-\delta_0}
+|\sigma_0|\nonumber\\
&\leq&C_{\eta 0}\nu^{-1}\delta_0^{-\tau}\|E_0\|_{\rho_0}\ ,
\eeqano
where
\beq{Ceta}
C_{\eta 0}\equiv C_{W0}\|M_0\|_{\rho_0}+C_{\sigma 0}\nu\delta_0^\tau\ .
\eeq
\end{remark}

\subsection{Estimates for the convergence of the iterative step}\label{sec:iterative}
In this Section we give quadratic estimates on the norm of $D\E[K_0,\mu_0] \Delta_0 + D_\mu \E[K_0,\mu_0] \sigma_0 +E_0$
with $\Delta_0\equiv M_0 W_0$; these estimates are needed to bound the error of the new approximate solution as
it will be done in Section~\ref{sec:error}.

\begin{lemma}\label{lem:quadestimate}
We have the following estimate:
\beq{EE}
\|E_0+D\E[K_0,\mu_0] \Delta_0 + D_\mu \E[K_0,\mu_0]
\sigma_0\|_{{\rho_0-\delta_0}} \le C_c\ C_{W0}\,
\nu^{-1}\, \delta_0^{-1-\tau} \|E_0\|^2_{\rho_0}\ .
\eeq
\end{lemma}

\begin{proof}
Taking into account that $W_0=M_0^{-1}\Delta_0$, from the relation
$(7.15)$ in \cite{CallejaCL11} we have that,
$$
E_0 + D\E[K_0,\mu_0] \Delta_0 + D_\mu \E[K_0,\mu_0] \sigma_0=R_0 W_0\ .
$$
From Lemma~\ref{lem:stepestimates} and \equ{R0}, we obtain that
\beqano
\|E_0+D\E[K_0,\mu_0] \Delta_0 + D_\mu \E[K_0,\mu_0] \sigma_0\|_{\rho_0-\delta_0}
&\leq& \|R_0\|_{\rho_0-\delta_0}\ \|W_0\|_{\rho_0-\delta_0}\nonumber\\
&\leq& C_c C_{W0}\nu^{-1}\delta_0^{-1-\tau}\|E_0\|_{\rho_0}^2\ .
\eeqano
In conclusion, we have \equ{EE}.
\end{proof}

\subsection{Estimates for the error of the new solution}\label{sec:error}
We proceed to bound the error corresponding to the new approximate
solution.

\begin{lemma}\label{lem:composition}
Let $\eta[K_0, \mu_0]$ be as in Lemma \ref{lem:quadestimate} and let
$\zeta >0$ be such that \beqano {\rm dist}( \mu_0, \partial
\Lambda)
&\ge&  \zeta\ , \\ 
{\rm dist}(K_0(\torus_{\rho_0}), \partial\mathcal{C})  &\ge&
 \zeta\ . 
\eeqano
Assume that
\beq{comp}
C_{\eta 0}\ \nu^{-1}\delta_0^{-\tau} \|E_0\|_{\rho_0}  < \zeta < 1
\eeq
with $C_{\eta 0}$ as in \equ{Ceta}.
Then, we obtain the following estimate for the error:
\[\|\E[K_0 + \Delta_0, \mu_0 + \sigma_0]\|_{{\rho_0-\delta_0}} \le C_{\E 0}\,
\nu^{-2}\, \delta_0^{-2\tau} \|E_0\|^2_{\rho_0}\ ,
\]
where
\beq{Ceps'}
C_{\E 0}\equiv C_c\ C_{W0}\nu
\delta_0^{-1+\tau}+C_{\mathcal{R} 0}
\eeq
with
\beq{CR}
C_{\mathcal{R} 0}\equiv
Q_{E 0}(\|M_0\|^2_{\rho_0}C_{W0}^2+C_{\sigma 0}^2\nu^2\delta_0^{2\tau})\ .
\eeq
\end{lemma}

\begin{proof}
Define the remainder of the Taylor series expansion as
$$
\mathcal{R} [(K_0, \mu_0), (K_0',
\mu_0')]\equiv \E[K_0',\mu_0']-\E[K_0,\mu_0]-D\E[K_0,\mu_0]
(K_0'-K_0)-D_\mu\E[K_0,\mu_0](\mu_0'-\mu_0)\ .
$$
Then, we can write
$$
\E[K_0+\Delta_0, \mu_0+\sigma_0] = E_0+ D \E[K_0, \mu_0] \Delta_0 + D_\mu \E[K_0,
\mu_0]\sigma_0 + \mathcal{R} [(K_0, \mu_0), (K_0+\Delta_0, \mu_0 + \sigma_0)]\ .
$$
From Lemma~\ref{lem:stepestimates} and the definition of $Q_{E 0}$ in \equ{QQQ}, we obtain
\beqano
\|\mathcal{R}\|_{{\rho_0-\delta_0}}\leq Q_{E 0}\
\Big(\|\Delta_0\|^2_{{\rho_0-\delta_0}}+|\sigma_0|^2\Big)
&\leq& Q_{E 0}\ \Big[\|M_0\|^2_{\rho_0}\ (C_{W0}\nu^{-1}\delta_0^{-\tau}\|E_0\|_{\rho_0})^2+(C_{\sigma 0}\|E_0\|_{\rho_0})^2\Big]\nonumber\\
&\equiv&C_{\mathcal{R} 0}
\nu^{-2}\,\delta_0^{-2\tau}\,\|E_0\|_{\rho_0}^2\ ,
\eeqano
with $C_{\mathcal{R} 0}$ as in \equ{CR}. Then, from Lemma~\ref{lem:quadestimate} we conclude that
\beqano
\|\E[K_0+\Delta_0, \mu_0+\sigma_0]\|_{{\rho_0-\delta_0}}&\leq&
C_c\ C_{W0} \nu^{-1}\delta_0^{-1-\tau}\
\|E_0\|_{\rho_0}^2
+C_{\mathcal{R} 0}\nu^{-2}\delta_0^{-2\tau}\ \|E_0\|_{\rho_0}^2\nonumber\\
&\leq& C_{\E 0}\ \nu^{-2}\delta_0^{-2\tau}\
\|E_0\|_{\rho_0}^2
\eeqano
with $C_{\E 0}$ as in \equ{Ceps'}. Notice that \equ{comp} guarantees that
$$
\|\Delta_0\|_{\rho_0-\delta_0}<\zeta\ ,\qquad |\sigma_0|<\zeta\ .
$$
\end{proof}

\subsection{Analytic convergence}\label{sec:analytic}
In this Section we prove that
if we start with a small enough error, it is possible to repeat indefinitely the algoritm
and that iterating the algorithm, we obtain a sequence of approximate
solutions which converge to the true solution of the invariance
equation \equ{invariance}.

Again, let $(K_0,\mu_0)$ be the initial approximate solution
with $K_0\in \A_{\rho_0}$ for some $\rho_0>0$ as in Theorem~\ref{main} and define
the sequence of parameters $\{\delta_h\}$, $\{\rho_h\}$, $h\geq 0$, as
$$
\delta_h\equiv{\rho_0\over 2^{h+2}}\ ,\qquad
\rho_{h+1}\equiv\rho_h-\delta_h\ ,\qquad h\geq 0\ .
$$
With this choice of parameters the domain of analyticity where the true solution is defined will be
a non--empty domain with size $\rho_\infty$ given by
$$
\rho_\infty=\rho_0-\sum_{j=0}^\infty {\rho_0\over
2^{j+2}}=\rho_0-{\rho_0\over 2}>0\ .
$$
Let $(K_h,\mu_h)$, $h\geq 1$, be the approximate solution constructed by
finding at each step the corrections $(W_h,\sigma_h)$ solving the analogous of the
cohomological equations \equ{ceq} for $h=0$. To make the notation precise, all quantities associated to $(K_h,\mu_h)$ will
carry a subindex $h$, indicating the step of the algorithm. Define
$$
\varepsilon_h \equiv \|\E(K_h, \mu_h)\|_{\rho_h}\ ,
$$
and let us introduce the following quantities:
$$
d_h \equiv \|\Delta_h\|_{\rho_h}, \quad v_h \equiv
\|D\Delta_h\|_{\rho_h}, \quad s_h \equiv |\sigma_h|\ .
$$
By Lemma \ref{lem:stepestimates} we have the following
inequalities:
\beqano
d_h &\leq& C_{dh} \nu^{-1} \delta_h^{-\tau}\,\varepsilon_h\ ,\nonumber\\
v_h &\leq& C_{dh} \ C_c\ \nu^{-1} \delta_h^{-\tau-1}\,\varepsilon_h\ ,\nonumber\\
s_h &\leq& C_{\sigma h} \varepsilon_h\ ,
\eeqano
where
\beq{Chath}
C_{dh}\equiv C_{Wh}\|M_h\|_{\rho_h}
\eeq
where the quantities $C_{Wh}$, $C_{\sigma h}$ are obtained through the following expressions:
\beqa{Csigmah}
C_{\sigma h}&\equiv& \mathcal{T}_h\ \Big[|\lambda-1|\Big({1\over {||\lambda|-1|}}\|S_h\|_{\rho_h}+1\Big)
+\|S_h\|_{\rho_h}\Big]\ \|M_h^{-1}\|_{\rho_h}\ ,\nonumber\\
\overline{C}_{W_2h}&\equiv&2\mathcal{T}_h \Big({1\over {||\lambda|-1|}}\|S_h\|_{\rho_h}+1\Big)\, Q_{\mu h}\|M_h^{-1}\|^2_{\rho_h}\ ,\nonumber\\
C_{W_2h}&\equiv&{1\over {||\lambda|-1|}}\Big(1+C_{\sigma h}\ Q_{\mu h}\Big)\|M_h^{-1}\|_{\rho_h}\ ,\nonumber\\
C_{W_1h}&\equiv&C_0\Big[\|S_h\|_{\rho_h}(C_{W_2h}+\overline{C}_{W_2h})+\|M_h^{-1}\|_{\rho_h}+
Q_{\mu h} \|M_h^{-1}\|_{\rho_h}\ C_{\sigma h}\Big]\ ,\nonumber\\
C_{Wh}&\equiv&C_{W_1 h}+(C_{W_2 h}+\overline{C}_{W_2h})\nu\delta_h^{\tau}\ .
\eeqa

\begin{remark}\label{rmk:constants}
By Lemma \ref{lem:composition} one has
$$
\varepsilon_{h+1} \leq  C_{\E h} \nu^{-2}
\delta_{h}^{-2\tau} \varepsilon_{h}^2\ ,
$$
where $C_{\E h}$ is defined as
\beq{Cepsp}
C_{\E h}\equiv C_c\ C_{W h} \nu \delta_{h}^{-1+\tau}+C_{\mathcal{R}h}
\eeq
with
\beq{cr}
C_{\mathcal{R}h}\equiv Q_{Eh}\
(\|M_h\|^2_{\rho_h}C_{Wh}^2+C_{\sigma h}^2\nu^2\delta_h^{2\tau})
\eeq
and
$$
Q_{Eh}\equiv {1\over 2}\max\Big\{\|D^2E_h\|_{{\rho_h-\delta_h}},\|DD_\mu E_h\|_{{\rho_h-\delta_h}},
\|D^2_\mu E_h\|_{{\rho_h-\delta_h}}\Big\}\ .
$$
\end{remark}

The results of Theorem~\ref{main} are based on the following proposition.

\begin{proposition}\label{pro:p1p2p3}
Let the constants
$C_{d0}$, $C_{\sigma 0}$, $C_{\E 0}$ be as in \equ{Chath}, \equ{Csigmah}, \equ{Cepsp} with $h=0$.
Define the following quantities:
\beq{kappa}
\kappa_K\equiv 4 C_{d0}\nu^{-1}\delta_0^{-\tau}\ ,\qquad \kappa_\mu\equiv 4 C_{\sigma 0}\ , \qquad
\kappa_0\equiv 2^{2\tau+1}\ C_{\E 0} \nu^{-2}\delta_0^{-2\tau}\ .
\eeq

Assume that the following conditions are satisfied:
\beq{conda}
2^{\tau+3}\,\kappa_0\varepsilon_0\leq 1\ ,
\eeq
\beq{condc}
\kappa_K \varepsilon_0<\zeta\ ,
\eeq
\beq{condd}
\kappa_\mu\varepsilon_0<\zeta\ ,
\eeq
\beq{condb}
\|N_0\|_{\rho_0}\ (2\|DK_0\|_{\rho_0}+D_K)\ D_K<1
\eeq
\beq{conde}
C_\sigma\ D_K\leq C_{\sigma0}\ ,
\eeq
\beq{condf}
D_K\ \Big(C_W C_c \nu \delta_0^{-1+\tau}+C_{\mathcal{R}}\Big) \leq C_{\E 0}\ ,
\eeq
\beq{condg}
D_K(C_{W0}+\|M_0\|_{\rho_0}C_W+C_WD_K)\leq C_{d0}\ ,
\eeq
\beq{condk}
4 Q_{z\mu 0} C_{\sigma 0}\varepsilon_0<Q_0\ ,
\eeq
\beq{condm}
4 Q_{\mu\mu 0} C_{\sigma 0}\varepsilon_0<Q_{\mu 0}\ ,
\eeq
where the constants $C_\sigma$, $C_W$, $C_{\mathcal{R}}$, $C_{W0}$,
$D_K$ are defined in Appendix~\ref{app:constants}.
Then, for all integers $h\geq 0$ the following inequalities $(p1;h)$, $(p2;h)$, $(p3;h)$ hold:
\begin{itemize}
\item[$(p1;h)$]
\beqa{Kmu}
\|K_h - K_0\|_{\rho_h} &\leq& \kappa_K\, \varepsilon_0<\zeta\ ,\nonumber\\
|\mu_h-\mu_0|&\leq & \kappa_\mu\, \varepsilon_0 < \zeta\ ;
\eeqa
\item[$(p2;h)$]
$$
\varepsilon_h \leq (\kappa_0 \varepsilon_0)^{2^h-1}\varepsilon_0\ ;
$$
\item[$(p3;h)$]
\beqa{p3h}
C_{dh} &\leq& 2 C_{d0}\ ,\nonumber\\
C_{\sigma h}&\leq& 2 C_{\sigma 0}\ ,\nonumber\\
C_{\E h}&\leq& 2 C_{\E 0}\ .
\eeqa
\end{itemize}
\end{proposition}

The proof of Proposition~\ref{pro:p1p2p3} is quite long (see Section~\ref{sec:proofp1p2p3}),
but it is well structured and broken into small steps that can be easily verified.
However, it allows to give the proof of Theorem~\ref{main} by \sl analytic smoothing: \rm at each step, the
corrections $(W_h,\sigma_h)$ allow to construct increasingly approximate solutions, defined on smaller
analyticity domains. The loss of domain is such that the exact solution is defined on a domain
with positive radius of analyticity.

\begin{proof} (of Theorem~\ref{main})
The inequalities \equ{kemu} follow directly from \equ{Kmu} and \equ{kappa}.
The condition \equ{C1} follows from \equ{comp} of Lemma~\ref{lem:composition},
while the conditions \equ{C2}-\equ{C10} follow from \equ{conda}-\equ{condm}
of Proposition~\ref{pro:p1p2p3}.
\end{proof}

\vskip.1in

\section{Proof of Proposition~\ref{pro:p1p2p3}}\label{sec:proofp1p2p3}
The proof of Proposition~\ref{pro:p1p2p3} proceeds by induction.
We start by noticing that $(p1;0)$, $(p2;0)$ and $(p3;0)$ are trivial.
Let $H\in\integer_+$ and assume that $(p1;h)$, $(p2;h)$, $(p3;h)$ are true for $h=1,...,H$.
Then, by Lemma  \ref{lem:composition} we obtain the Taylor
estimate
\beqa{compest}
\varepsilon_{h} &=&
\|\E(K_{h-1}+\Delta_{h-1}, \mu_{h-1}+\sigma_{h-1})\|_{{\rho_{h}}}\nonumber\\
&\leq&  C_{\E,h-1} \nu^{-2} \delta_{h-1}^{-2\tau}
\varepsilon_{h-1}^2\nonumber\\
&\leq& 2 C_{\E 0} \nu^{-2}
\delta_{h-1}^{-2\tau} \varepsilon_{h-1}^2\ ,
\eeqa
where
$$
C_{\E,h-1}\leq 2 C_{\E 0}\ ,
$$
due to $(p3;h)$ for $h=1,...,H$. The estimate \equ{compest} allows to
have a bound of $\varepsilon_h$, $h=1,...,H$, in terms of $\varepsilon_0$:
\beqano
\varepsilon_{h}&\leq &2 C_{\E 0}\ \nu^{-2}\delta_0^{-2\tau}\ 2^{2\tau (h-1)} \varepsilon_{h-1}^2\nonumber\\
&\leq&(2 C_{\E 0} \nu^{-2}\delta_0^{-2\tau})\
2^{2\tau(h-1)}(2 C_{\E 0} \nu^{-2}\delta_0^{-2\tau}
\ 2^{2\tau (h-2)}\varepsilon_{h-2}^2)^2\nonumber\\
&\leq&(2 C_{\E 0}
\nu^{-2}\delta_0^{-2\tau})^{1+2+...+2^{h-1}}\
2^{2\tau((h-1)+2(h-2)+...+2^{h-2})}
\varepsilon_0^{2^h}\nonumber\\
&\leq&(2 C_{\E 0} \nu^{-2}\delta_0^{-2\tau})^{2^h-1}\ 2^{2\tau(2^h-(h+1))}\varepsilon_0^{2^h}\nonumber\\
&\leq& (2 C_{\E 0}
\nu^{-2}\delta_0^{-2\tau}2^{2\tau}\varepsilon_0)^{2^h-1}\varepsilon_0\ .
\eeqano

In Sections~\ref{sec:p1}, \ref{sec:p2}, \ref{sec:p3}, we will prove
$(p1;H+1)$, $(p2;H+1)$ and $(p3;H+1)$, assuming
the induction assumption $(p1;h)$, $(p2;h)$, $(p3;h)$ for $h=1,...,H$.
To get such result, we need the following Lemma, which gives bounds on
the quantities entering the estimates needed to prove the inductive assumption.

\vskip.1in

\begin{lemma}\label{lem:DK}
Assume that $(p1;h)$, $(p2;h)$, $(p3;h)$ hold for $h=1,...,H$.
For $H\in\integer_+$ the following inequality holds:
\beq{DKin}
\|DK_{H+1}-DK_0\|_{\rho_{H+1}}\leq D_K\ ,
\eeq
where (see \equ{DK1})
$$
D_K\equiv 4C_{d0}\ C_c\ \nu^{-1} \delta_0^{-\tau-1}\varepsilon_0\ ,
$$
where $C_{d0}$ is as in \equ{Chath} and provided that
\beq{condDK1}
2^{\tau+1}\kappa_0 \varepsilon_0\leq{1\over 2}
\eeq
with $\kappa_0$ as in \equ{kappa}.
Furthermore, under the inequality:
\beq{condDK2}
\|N_0\|_{\rho_0}\ (2\|DK_0\|_{\rho_0}+D_K)\ D_K < 1\ ,
\eeq
the following relations hold for $0\leq h\leq H+1$:
\beq{auxa}
    \|N_h - N_0\|_{\rho_h} \leq C_N D_K\ ,
\eeq
\beq{auxb}
    \|M_h - M_0\|_{\rho_h} \leq C_M D_K\ ,
\eeq
\beq{auxc}
    \|M^{-1}_h - M^{-1}_0\|_{\rho_h}  \leq C_{Minv} D_K\ ,
\eeq
where $C_N$, $C_M$, $C_{Minv}$ are defined as follows:
\beqa{constants}
C_N&\equiv&\|N_0\|_{\rho_0}^2\ {{2\|DK_0\|_{\rho_0}+D_K}\over {1-\|N_0\|_{\rho_0}D_K(2\|DK_0\|_{\rho_0}+D_K)}}\ ,\nonumber\\
C_M&\equiv& 1 + J_e\Big[C_N \Big(\|DK_0\|_{\rho_0}+D_K\Big)+\|N_0\|_{\rho_0}\Big]\ ,\nonumber\\
C_{Minv}&\equiv& C_{N}(\|DK_0\|_{\rho_0}+D_K)+\|N_0\|_{\rho_0}+J_e\ ,
\eeqa
and where $J_e$ is an upper bound on the norm of the matrix $J$ in \equ{matrixJ}.
\end{lemma}

\begin{proof}
We start by proving \equ{DKin}:
\beqa{DK}
\|DK_{H+1} - DK_0 \|_{\rho_{H+1}}&\leq& \sum_{j=0}^{H} v_j\nonumber\\
&\leq&
\sum_{j=0}^{H} C_{dj}\ C_c\ \nu^{-1}
\delta_j^{-\tau-1}\varepsilon_j\nonumber\\
&\leq& D_K
\eeqa
with $D_K$ as in \equ{DK1} and provided that \equ{condDK1} holds.

\vskip.1in

The proof of \equ{auxa} is obtained as follows. From the relations
\beqa{ST}
DK_h&=&DK_0+\widetilde K_h\ ,\nonumber\\
DK_h^\top&=&DK_0^\top+\widetilde K_h^\top\ ,\nonumber\\
\widetilde K_h&=&\sum_{j=0}^{h-1} D\Delta_j\ ,
\eeqa
we obtain
\beqano
N_h&=&(DK_h^\top DK_h)^{-1}=\Big((DK_0^\top+\widetilde K_h^\top)(DK_0+\widetilde K_h)\Big)^{-1}\nonumber\\
&=&(DK_0^\top DK_0+\widetilde K_h^\top DK_0+DK_0^\top\widetilde K_h+\widetilde K_h^\top\widetilde K_h)^{-1}\nonumber\\
&=&(DK_0^\top DK_0)^{-1}\Big(1+(DK_0^\top DK_0)^{-1}(\widetilde K_h^\top DK_0+DK_0^\top \widetilde K_h+\widetilde K_h^\top \widetilde K_h)\Big)^{-1}\nonumber\\
&=&N_0(1+\chi_h)^{-1}\ ,
\eeqano
having set $\chi_h\equiv N_0(\widetilde K_h^\top DK_0+DK_0^\top \widetilde K_h+\widetilde K_h^\top \widetilde K_h)$.
Under the inequality \equ{condDK2}, ensuring that $\|\chi_h\|_{\rho_h}<1$ and using \equ{DK}, we have the following bound:
$$
\|(1+\chi_h)^{-1}-1\|\leq {{\|\chi_h\|}\over {1-\|\chi_h\|}}\ ,
$$
which leads to
\beqa{Nh}
\|N_h-N_0\|_{\rho_h}&\leq&\|N_0\|_{\rho_0}\ \|(1+\chi_h)^{-1}-1\|_{\rho_h}\nonumber\\
&\leq&\|N_0\|_{\rho_0}\ {{\|\chi_h\|_{\rho_h}}\over {1-\|\chi_h\|_{\rho_h}}}\nonumber\\
&\leq&C_N\ D_K
\eeqa
with $C_N$ as in \equ{constants}.

\vskip.1in

The proof of \equ{auxb} is obtained starting from the identity
$$
M_h-M_0=(DK_h-DK_0\ |\ J^{-1}\circ K_h\ DK_h\ N_h-J^{-1}\circ K_0\
DK_0\ N_0)\ .
$$
Then, one has
$$
\|M_h-M_0\|_{\rho_h}\leq \|DK_h-DK_0\|_{\rho_h}+ J_e\ \|DK_h N_h-DK_0 N_0\|_{\rho_h}\ .
$$
From
$$
DK_h N_h-DK_0 N_0=DK_h N_h-DK_h N_0+DK_h N_0-DK_0 N_0
$$
and from \equ{Nh}, we obtain that
\beqa{DKN}
\|DK_h N_h-DK_0 N_0\|_{\rho_h}&\leq&\|DK_h\|_{\rho_h}\|N_h-N_0\|_{\rho_h}+\|N_0\|_{\rho_0}\|DK_h-DK_0\|_{\rho_h}\nonumber\\
&\leq&(C_N\|DK_h\|_{\rho_h}+\|N_0\|_{\rho_0})\ D_K\ .
\eeqa
Finally, we have\footnote{Notice that we can bound
$\|DK_h\|_{\rho_h}$ as $\|DK_h\|_{\rho_h}\leq \|DK_0\|_{\rho_0}+\|DK_h-DK_0\|_{\rho_h}$.}
$$
\|M_h-M_0\|_{\rho_h}\leq C_M\ D_K
$$
with $C_M$ as in \equ{constants}.

\vskip.1in

The proof of \equ{auxc} is obtained as follows. We have that the inverse of the
matrix $M_h$ can be written as\footnote{The matrix $M_h^{-1}$ is given by taking
the transpose of the matrix obtained juxtaposing the $2n\times n$ matrices (in the
generic $n$--dimensional case) $DK_h N_h^\top$ and $(J\circ K_h)^\top DK_h$ or,
equivalently, by constructing the matrix whose first $n$ rows are given by the
$n\times 2n$ matrix $N_h DK_h^\top$ and the second $n$ rows are given by the
$n\times 2n$ matrix $DK_h^\top (J\circ K_h)$.}

\beq{MMinv}
M_h^{-1}(\theta) = \Big(DK_h\ N_h^\top |\   (J\circ K_h)^\top\
DK_h\Big)^\top\ = \left(\begin{array}{c}
N_h DK_h^\top \\
DK_h^\top (J\circ K_h)\
 \end{array} \right).\nonumber\\
\eeq
Indeed, one can verify that due to the Lagrangian character, $M_h^{-1} M_h = \Id$.
By computing the inverse of $M_0$ in an analogous way, one has
\beq{MM}
M_h^{-1}-M_0^{-1}= \left(\begin{array}{c}
N_h DK_h^\top - N_0 DK_0^\top \\
DK_h^\top (J\circ K_h)\ - DK_0^\top (J\circ K_0)\
 \end{array} \right)\ .
\eeq
The bound for the first row $N_h DK_h^\top - N_0 DK_0^\top$ is obtained as in \eqref{DKN}, while
the second row is bounded by $J_e\ D_K$. This yields \equ{auxc} with $C_{Minv}$
as in \equ{constants}.

\end{proof}

\vskip.1in

We are now in the position to continue with the proof of $(p1;H+1)$, $(p2;H+1)$, $(p3;H+1)$ to which we devote
the rest of this Section.

\vskip.1in

\subsection{Proof of $(p1;H+1)$}\label{sec:p1}
Using the inequality $j+1\leq 2^j$, one has
\beqano
\|K_{H+1} - K_0\|_{{\rho_{H+1}}} &\leq& \sum_{j=0}^H d_j \leq \sum_{j=0}^H (C_{dj} \nu^{-1} \delta_j^{-\tau}\varepsilon_j)\nonumber \\
& \leq & 4 C_{d0}\nu^{-1} \delta_0^{-\tau}\varepsilon_0\ ,
\eeqano
assuming that $\varepsilon_0$ satisfies \equ{conda}.
In conclusion, we have
$$
\|K_{H+1} - K_0\|_{{\rho_{H+1}}} \leq \kappa_K \varepsilon_0
$$
with $\kappa_K$ as in \equ{kappa}. Moreover, we have:
\beqano
|\mu_{H+1} - \mu_0| &\leq& \sum_{j=0}^H s_j \leq \sum_{j=0}^H C_{\sigma_j} \varepsilon_j \nonumber\\
&\leq &2 C_{\sigma 0}\, \sum_{j=0}^H
(\kappa_0\varepsilon_0)^{2^j-1}\varepsilon_0\ ;
\eeqano
assuming that $\varepsilon_0$ satisfies \equ{conda}, we conclude that
$$
|\mu_{H+1} - \mu_0| \leq
4 C_{\sigma 0}\,\varepsilon_0 = \kappa_\mu\varepsilon_0\ ,
$$
with $\kappa_\mu$ as in \equ{kappa}.
We take $\varepsilon_0$ small enough so that \equ{condc} and \equ{condd} are satisfied, which provide $(p1;H+1)$.

\subsection{Proof of $(p2;H+1)$}\label{sec:p2}
Having proven $(p1;H+1)$,
we use the Taylor estimate \eqref{compest} with $H+1$ in place of
$h$ to obtain $(p2;H+1)$:
$$
\varepsilon_{H+1} \leq (2 C_{\E 0}
\nu^{-2}\delta_0^{-2\tau}2^{2\tau}\varepsilon_0)^{2^{H+1}-1}\varepsilon_0=
(\kappa_0\varepsilon_0)^{2^{H+1}-1}\varepsilon_0\ ,
$$
due to the definition of $\kappa_0$ in \equ{kappa}.

\subsection{Proof of $(p3;H+1)$}\label{sec:p3}
The proof of $(p3;H+1)$ is rather cumbersome and needs several auxiliary results.
Given the inductive assumption, we want to prove that
\beq{ccc}
C_{d,H+1}\leq 2C_{d0}\ ,
\qquad C_{\sigma,H+1}\leq 2 C_{\sigma 0}\ ,\qquad
C_{\E,H+1}\leq 2C_{\E 0}\ .
\eeq
First, we estimate $|{\mathcal T}_h - {\mathcal T}_0|$ as described in Section~\ref{sec:TT}.

\subsubsection{Estimate on $|{\mathcal T}_h - {\mathcal T}_0|$}\label{sec:TT}

Before describing the proof of $(p3;H+1)$ we need the following auxiliary result.

\begin{lemma}\label{lem:DKT}
Assume that $(p1;h)$, $(p2;h)$, $(p3;h)$ hold for $h=1,...,H$ and
that the condition \equ{condDK2} of Lemma~\ref{lem:DK} is valid together with
\beqa{CQ2}
4 Q_{z\mu 0} C_{\sigma 0}\varepsilon_0&<&Q_0\ ,\nonumber\\
4 Q_{\mu\mu 0} C_{\sigma 0}\varepsilon_0&<&Q_{\mu 0}\ .
\eeqa
Let ${\mathcal T}_0$, ${\mathcal T}_h$ be defined as
$$
{\mathcal T}_0\equiv \|\tau_0\|_{\rho_0}\ ,\qquad
{\mathcal T}_h\equiv \|\tau_h\|_{\rho_h}\ .
$$
For $h\in\nat$, $h=1,...,H$, the following inequality holds:
\beq{auxd}
    |{\mathcal T}_h - {\mathcal T}_0| \leq C_T D_K\ ,
\eeq
where $C_T$ is defined as
\beq{CT}
C_T\equiv {{\mathcal{T}_0^2}\over {1-\mathcal{T}_0C_\tau}}\
\max\Big\{C_S,C_{SB}+2C_{Minv} Q_{\mu 0}\Big\}
\eeq
with
\beqano
C_S&\equiv& 2 J_e\ Q_0\
\Big\{\Big(\|N_0\|_{\rho_0}+C_N D_K\Big)\ \Big[D_K
(\|N_0\|_{\rho_0}+C_N D_K)\nonumber\\
&+&\|DK_0\|_{\rho_0}\|N_0\|_{\rho_0}+\|DK_0\|_{\rho_0} C_N D_K\Big]+
C_N\ \|DK_0\|_{\rho_0}\ \Big[D_K (\|N_0\|_{\rho_0}+C_N D_K)\nonumber\\
&+&\|DK_0\|_{\rho_0}\|N_0\|_{\rho_0}+\|DK_0\|_{\rho_0}C_N D_K\Big]+
\|N_0\|_{\rho_0}\|DK_0\|_{\rho_0}(\|N_0\|_{\rho_0}+C_N D_K)\nonumber\\
&+&C_N \|N_0\|_{\rho_0}\|DK_0\|_{\rho_0}^2\Big\}\ ,\nonumber\\
C_{SB}&\equiv& {1\over {||\lambda|-1|}}Q_{\mu 0}\|M_0^{-1}\|_{\rho_0}C_S+
2J_e Q_0\ \|N_0\|_{\rho_0}^2\ \|DK_0\|_{\rho_0}^2{1\over {||\lambda|-1|}}\ C_{Minv}
\ Q_{\mu 0}\nonumber\\
&+&2C_S\ {1\over {||\lambda|-1|}}\ C_{Minv}\ Q_{\mu 0}\ D_K\ ,\nonumber\\
C_\tau&\equiv&\max\Big\{C_S,C_{SB}+2C_{Minv} Q_{\mu 0}\Big\}\ D_K\ .
\eeqano
\end{lemma}

\begin{proof}
Let $\tau_0$ and $\tau_h$ be defined as
$$
\tau_0\equiv\left(%
\begin{array}{cc}
  {\overline S}_0 & {\overline {S_0(B_{b0})^0}}+\overline{\widetilde A_0^{(1)}} \\
  \lambda-1 & \overline{\widetilde A_0^{(2)}} \\
 \end{array}%
\right)^{-1}
$$
and
$$\tau_h\equiv
\left(%
\begin{array}{cc}
  {\overline S}_h & {\overline {S_h(B_{bh})^0}}+\overline{\widetilde A_h^{(1)}} \\
  \lambda-1 & \overline{\widetilde A_h^{(2)}}\\
 \end{array}%
\right)^{-1}\ .
$$
Then, we obtain
$$
{\mathcal T}_h \leq \|\tau_0\|_{\rho_0}+\|\widetilde\tau_h\|_{\rho_h}={\mathcal T}_0+\|\widetilde\tau_h\|_{\rho_h}\ ,
$$
where $\widetilde\tau_h\equiv \tau_h-\tau_0=\tau_0^2\
\Big[\Big(I+\tau_0(\tau_h^{-1}-\tau_0^{-1})\Big)^{-1}\ (\tau_0^{-1}-\tau_h^{-1})\Big]$, so that we have the estimate
\beq{T}
|{\mathcal T}_h-{\mathcal T}_0|\leq \|\widetilde\tau_h\|_{\rho_h}
\eeq
with
\beq{tau}
\|\widetilde\tau_h\|_{\rho_h} \leq {{{\mathcal
T}_0^2}\over {1-{\mathcal T}_0 C_\tau}}C_\tau\ ,
\eeq
where $C_\tau$ is a bound on $\tau_h^{-1}-\tau_0^{-1}$, say
\beq{tauh}
\|\tau_h^{-1}-\tau_0^{-1}\|_{\rho_h} \equiv \left\|\left(%
\begin{array}{cc}
  {\overline S}_h-{\overline S}_0 & {\overline {S_h(B_{bh})^0}}+\overline{\widetilde A_h^{(1)}}-
  ({\overline {S_0(B_{b0})^0}}+\overline{\widetilde A_0^{(1)}}) \\
  0 & \overline{\widetilde A_h^{(2)}}-\overline{\widetilde A_0^{(2)}} \\
 \end{array}%
\right)\right\|_{\rho_h}\leq C_\tau\ .
\eeq
To obtain an expression for $C_\tau$, we bound term by term the matrix appearing
in \equ{tauh}.

We start to estimate the first element of the matrix appearing in \equ{tauh}, namely $\|{\overline S}_h-{\overline
S}_0\|_{\rho_h}$. From \equ{def} we have that $S_h$ is defined by
$$
S_h=N_h(\theta+\omega)^\top  DK_h(\theta+\omega)^\top \ Df_{\mu_h}\circ K_h(\theta)\
J^{-1}\circ K_h(\theta)\ DK_h(\theta)N_h(\theta)\ .
$$
Then, we bound $Df_{\mu_h}\circ K_h$ with
\beqano
\sup_{z\in{\C}}|Df_{\mu_h}(z)|&\leq&\sup_{z\in{\C}}|Df_{\mu_0}(z)|+\sup_{z\in{\C}}|Df_{\mu_h}(z)-Df_{\mu_0}(z)|\nonumber\\
&\leq&Q_0+\sup_{z\in{\C},\mu\in\Lambda,|\mu-\mu_0|<2\kappa_\mu\varepsilon_0}|D_\mu Df_{\mu}(z)|\ |\mu_h-\mu_0|\nonumber\\
&\leq& Q_0+4 Q_{z\mu 0}\ C_{\sigma 0}\varepsilon_0\leq 2Q_0\ ,
\eeqano
if \equ{condk} holds. Notice that we have used $(p1;H+1)$ to bound $\mu_h-\mu_0$ for $h=1,...,H+1$.
Finally, we obtain
\beqano
\|S_h-S_0\|_{\rho_h}&\leq&2Q_0\ \|N_h(\theta+\omega)^\top  DK_h(\theta+\omega)^\top
\ J^{-1}\circ K_h(\theta)\ DK_h(\theta)N_h(\theta)\nonumber\\
&-&N_0(\theta+\omega)^\top  DK_0(\theta+\omega)^\top \ J^{-1}\circ
K_0(\theta)\ DK_0(\theta)N_0(\theta)\|_{\rho_h}\ .
\eeqano
Setting $\widetilde N_h=N_h-N_0$ and writing $DK_h$ as $DK_h=DK_h-DK_0+DK_0$, one obtains
\beqano
\|S_h-S_0\|_{\rho_h}&\leq&2Q_0\ \|(N_0+\widetilde N_h)(\theta+\omega)^\top (DK_h-DK_0+DK_0)(\theta+\omega)^\top \nonumber\\
&&J^{-1}\circ K_h(\theta)\ (DK_h-DK_0+DK_0)(\theta)(N_0+\widetilde N_h)(\theta)\nonumber\\
&-&N_0(\theta+\omega)^\top  DK_0(\theta+\omega)^\top  \ J^{-1}\circ
K_0(\theta)\ DK_0(\theta)N_0(\theta)\|_{\rho_h}\ .
\eeqano
Let us bound $\widetilde N_h$ using \eqref{auxa}.
Then, using that $J$ is a constant matrix, we have:
\beqano
\|S_h-S_0\|_{\rho_h}&\leq&2Q_0\
\|[((N_0+
\widetilde N_h)\circ T_\omega)^\top \,((DK_h-DK_0)\circ T_\omega)^\top \nonumber\\
&\qquad+&(N_0\circ T_\omega)^\top  (DK_0\circ T_\omega)^\top +(\widetilde N_h\circ T_\omega)^\top  (DK_0\circ T_\omega)^\top ]\ J^{-1}\circ K_h(\theta)\nonumber\\
&&[(DK_h-DK_0)(N_0+\widetilde N_h)+DK_0 N_0+DK_0\widetilde N_h]\nonumber\\
&\qquad-&(N_0\circ T_\omega)^\top  (DK_0\circ T_\omega)^\top \ J^{-1}\circ K_0(\theta)\ DK_0N_0\|_{\rho_h}\nonumber\\
&=&2Q_0\ \|((N_0+\widetilde N_h)\circ T_\omega)^\top ((DK_h-DK_0)\circ T_\omega)^\top \nonumber\\
&&J^{-1}\circ K_h(\theta)\ [(DK_h-DK_0)(N_0+\widetilde N_h)+DK_0 N_0+DK_0\widetilde N_h]\nonumber\\
&\qquad+&(\widetilde N_h\circ T_\omega)^\top  (DK_0\circ T_\omega)^\top \nonumber\\
&&J^{-1}\circ K_h(\theta)\ [(DK_h-DK_0)(N_0+\widetilde N_h)+DK_0 N_0+DK_0\widetilde N_h]\nonumber\\
&\qquad+&(N_0\circ T_\omega)^\top  (DK_0\circ T_\omega)^\top \ J^{-1}\circ K_h(\theta)\ (DK_h-DK_0)(N_0+\widetilde N_h)\nonumber\\
&\qquad+&(N_0\circ T_\omega)^\top  (DK_0\circ T_\omega)^\top \ J^{-1}\circ K_h(\theta)\ DK_0\widetilde N_h\|_{\rho_h}\nonumber\\
&\leq&2 J_e\ Q_0\ \Big\{(\|N_0\|_{\rho_0}+\|\widetilde N_h\|_{\rho_h})\ \|DK_h-DK_0\|_{\rho_h}\nonumber\\
&&\Big[\|DK_h-DK_0\|_{\rho_h}
(\|N_0\|_{\rho_0}+\|\widetilde N_h\|_{\rho_h})+\|DK_0\|_{\rho_0}\|N_0\|_{\rho_0}+\|DK_0\|_{\rho_0}\|\widetilde N_h\|_{\rho_h}\Big]\nonumber\\
&\qquad+&\|\widetilde N_h\|_{\rho_h}\|DK_0\|_{\rho_0}\Big[\|DK_h-DK_0\|_{\rho_h}(\|N_0\|_{\rho_0}+\|\widetilde N_h\|_{\rho_h})\nonumber\\
&\qquad+&\|DK_0\|_{\rho_0}\|N_0\|_{\rho_0}+\|DK_0\|_{\rho_0}\|\widetilde N_h\|_{\rho_h}\Big]\nonumber\\
&\qquad+&\|N_0\|_{\rho_0}\|DK_0\|_{\rho_0}\|DK_h-DK_0\|_{\rho_h}\ (\|N_0\|_{\rho_0}+\|\widetilde N_h\|_{\rho_h})\nonumber\\
&\qquad+&\|N_0\|_{\rho_0}\|DK_0\|_{\rho_0}^2\|\widetilde
N_h\|_{\rho_h}\Big\}\ .
\eeqano
Taking into account \equ{auxa}, \equ{constants}, we obtain:
\beq{S}
\|S_h-S_0\|_{\rho_h}\leq C_S\ D_K
\eeq
with
\beqa{CtildeS}
C_S&\equiv& 2J_e\ Q_0\
\Big\{(\|N_0\|_{\rho_0}+C_ND_K)\ \Big[D_K
(\|N_0\|_{\rho_0}+C_N D_K)\nonumber\\
&+&\|DK_0\|_{\rho_0}\|N_0\|_{\rho_0}+\|DK_0\|_{\rho_0} C_N D_K\Big]\nonumber\\
&+&C_N\|DK_0\|_{\rho_0}\ \Big[D_K(\|N_0\|_{\rho_0}+C_N D_K)
+\|DK_0\|_{\rho_0}\|N_0\|_{\rho_0}+\|DK_0\|_{\rho_0}C_N D_K\Big]\nonumber\\
&+&\|N_0\|_{\rho_0}\|DK_0\|_{\rho_0}(\|N_0\|_{\rho_0}+C_N D_K)\nonumber\\
&+&C_N \|N_0\|_{\rho_0} \|DK_0\|_{\rho_0}^2\Big\}\ .
\eeqa
Now we bound the upper right element of the matrix appearing in \equ{tauh}. This computation
will give us also a bound on the lower right element of the matrix in \equ{tauh}. We start from (see \equ{EA})
\beqano
\widetilde A_h&=&M_h^{-1}\circ T_\omega\ D_\mu f_{\mu_h}\circ K_h\ ,\nonumber\\
\widetilde A_0&=&M_0^{-1}\circ T_\omega\ D_\mu f_{\mu_0}\circ K_0
\eeqano
and the estimate:
$$
\sup_{z\in\C}|D_\mu f_{\mu_h}(z)|\leq Q_{\mu 0}+4Q_{\mu\mu 0}\ C_{\sigma_0}\varepsilon_0\leq 2Q_{\mu 0}\ ,
$$
provided \equ{condm} holds. Then, we have:
\beqano
\|\widetilde A_h-\widetilde A_0\|_{\rho_h}&\leq&2Q_{\mu 0}\ \|M_h^{-1}-M_0^{-1}\|_{\rho_h}\nonumber\\
&\leq&2C_{Minv}\ Q_{\mu 0}\
D_K\ .
\eeqano
Next we estimate $\|{\overline {S_h(B_{bh})^0}}-{\overline {S_0(B_{b0})^0}}\|_{\rho_h}$; recall
that from \equ{BB} we have that $(B_{bh})^0$ is the solution of \beq{Bbh} \lambda
(B_{bh})^0-(B_{bh})^0\circ T_\omega=-(\widetilde A_h^{(2)})^0\ , \eeq
while $(B_{b0})^0$ is the solution of \beq{B0} \lambda
(B_{b0})^0-(B_{b0})^0\circ T_\omega=-(\widetilde A_0^{(2)})^0\ . \eeq
Expanding \equ{Bbh} and \equ{B0} in Fourier series, we obtain
$$
\sum_{j\in\integer} (\widehat B_{bh})_j^0\ (\lambda-e^{2\pi ij\omega})\
e^{2\pi ij\theta}=-\sum_{j\in\integer} (\widehat A_h^{(2)})^0_j\ e^{2\pi ij\theta}\ ,
$$
so that
$$
(B_{bh})^0(\theta)=-\sum_{j\in\integer} {{(\widehat A_h^{(2)})^0_j}\over
{\lambda - e^{2\pi ij\omega}}}\ e^{2\pi ij\theta}\ ,
$$
and similarly
$$
(B_{b0})^0(\theta)=-\sum_{j\in\integer} {{(\widehat A_0^{(2)})^0_j}\over
{\lambda - e^{2\pi ij\omega}}}\ e^{2\pi ij\theta}\ .
$$
In conclusion we have:
\beq{BBh}
(B_{bh})^0(\theta)-(B_{b0})^0(\theta)=-\sum_{j\in\integer} {{(\widehat
A_h^{(2)})^0_j-(\widehat A_0^{(2)})^0_j}\over {\lambda - e^{2\pi ij\omega}}}\
e^{2\pi ij\theta}\ .
\eeq
From \equ{BBh}, let us write $(B_{bh})^0$ as
$$
(B_{bh})^0=(B_{b0})^0+\widetilde B_h\ ,
$$
where
$$
\widetilde B_h\equiv -\sum_{j\in\integer} {{(\widehat A_h^{(2)})^0_j-(\widehat
A_0^{(2)})^0_j}\over {\lambda - e^{2\pi ij\omega}}}\ e^{2\pi ij\theta}\ .
$$
Let us introduce
$$
\widetilde S_h\equiv S_h-S_0\ ,
$$
whose norm can be bounded by \equ{S}. Then, we have:
\beqano
\|{\overline {S_h(B_{bh})^0}}-{\overline
{S_0(B_{b0})^0}}\|&=&\|{\overline {(S_0+\widetilde S_h)\ ((B_{b0})^0+\widetilde B_h)}}-
{\overline {S_0(B_{b0})^0}}\|\nonumber\\
&=&\|{\overline {S_0(B_{b0})^0}}+{\overline {\widetilde S_h(B_{b0})^0}}+{\overline {S_0\widetilde B_h}}
+{\overline {\widetilde S_h \widetilde B_h}}-{\overline {S_0(B_{b0})^0}}\|\nonumber\\
&\leq&\|(B_{b0})^0\|_{\rho_0}\|\widetilde S_h\|_{\rho_h}+\|S_0\|_{\rho_0}
\|\widetilde B_h\|_{\rho_h}+\|\widetilde S_h\|_{\rho_h} \|\widetilde B_h\|_{\rho_h}\ , \eeqano
where
\beqano
\|S_0\|_{\rho_0}&\leq&J_e\ Q_0\|N_0\|_{\rho_0}^2\ \|DK_0\|_{\rho_0}^2\ ,\nonumber\\
\|\widetilde S_h\|_{\rho_h}&\leq&C_S\ D_K\ ,\nonumber\\
\|(B_{b0})^0\|_{\rho_0}&\leq&{1\over {||\lambda|-1|}}\|\widetilde A_0^{(2)}\|_{\rho_0}
\leq{1\over {||\lambda|-1|}}Q_{\mu 0}\ \|M_0^{-1}\|_{\rho_0}\ ,\nonumber\\
\|\widetilde B_h\|_{\rho_h}&\leq&{1\over {||\lambda|-1|}}\ 2C_{Minv}\ Q_{\mu 0}\ D_K\ .
\eeqano
Then, we have:
$$
\|{\overline {S_h(B_{bh})^0}}-{\overline {S_0(B_{b0})^0}}\|\leq
C_{SB}\ D_K\ ,
$$
where
\beqa{Ctilde2S}
C_{SB}&\equiv& {1\over {||\lambda|-1|}}Q_{\mu 0}\|M_0^{-1}\|_{\rho_0}C_S\nonumber\\
&+&2J_e\ Q_0 \|N_0\|_{\rho_0}^2\
\|DK_0\|_{\rho_0}^2{1\over {||\lambda|-1|}}\ C_{Minv}
\ Q_{\mu 0}\nonumber\\
&+&2C_S\ {1\over {||\lambda|-1|}}\ C_{Minv}\
Q_{\mu 0}\ D_K\ .
\eeqa
Recalling \equ{tauh}, we obtain
\beqa{taue}
\|\tau_h^{-1}-\tau_0^{-1}\|_{\rho_h}&\leq&\max\Big\{\|{\overline
{S}_h}-{\overline {S}_0}\|_{\rho_h}, \|{\overline
{S_h(B_{bh})^0}}-{\overline {S_0(B_{b0})^0}}\|_{\rho_h}+
\|{\overline{\widetilde A_h^{(1)}}}-{\overline{\widetilde
A_0^{(1)}}}\|_{\rho_h}
+\|{\overline{\widetilde A_h^{(2)}}}-{\overline{\widetilde A_0^{(2)}}}\|_{\rho_h}\Big\}\nonumber\\
&\leq&\max\Big\{C_S, C_{SB}+2C_{Minv}\ Q_{\mu 0}\Big\}\ D_K\equiv C_\tau\ ,
\eeqa
where $C_\tau$ is defined by the last inequality in \equ{taue}.
From \equ{T} and \equ{tau} we get \equ{auxd}:
$$
|\mathcal{T}_h-\mathcal{T}_0|\leq C_T\ D_K
$$
with
\beq{CT}
C_T\equiv {{\mathcal{T}_0^2}\over {1-\mathcal{T}_0 C_\tau}}\ \max\Big\{C_S, C_{SB}+2C_{Minv}\ \ Q_{\mu 0}\Big\}\ .
\eeq
\end{proof}

\subsubsection{Proof of $C_{\sigma,H+1}\leq 2C_{\sigma 0}$}
We now prove \equ{ccc} and we begin from the second inequality.
We start with the following relations, which are a consequence of \equ{Csigmah}:
\beqano
C_{\sigma,H+1}&=&\mathcal{T}_{H+1}\Big[|\lambda-1|\Big({1\over {||\lambda|-1|}}\|S_{H+1}\|_{\rho_{H+1}}+1\Big)
+\|S_{H+1}\|_{\rho_{H+1}}\Big]\ \|M_{H+1}^{-1}\|_{\rho_{H+1}}\ ,\nonumber\\
C_{\sigma 0}&=&\mathcal{T}_0 \Big[|\lambda-1|\ \Big({1\over {||\lambda|-1|}}\|S_0\|_{\rho_0}+1\Big)+
\|S_0\|_{\rho_0}\Big]\ \|M_0^{-1}\|_{\rho_0}
\eeqano
with
\beqano
\|M_{H+1}^{-1}\|_{\rho_{H+1}}&\leq&\|M_0^{-1}\|_{\rho_0}+\|M_{H+1}^{-1}-M_0^{-1}\|_{\rho_{H+1}}\nonumber\\
&\leq&\|M_0^{-1}\|_{\rho_0}+C_{Minv}\ D_K
\eeqano
with $C_{Minv}$ as in \equ{constants}. We also have
\beqano
\|S_{H+1}\|_{\rho_{H+1}}&\leq&\|S_0\|_{\rho_0}+\|S_{H+1}-S_0\|_{\rho_{H+1}}\nonumber\\
&\leq&\|S_0\|_{\rho_0}+C_S\ D_K
\eeqano
with $C_S$ as in \equ{CtildeS}. From the relation
$$
\mathcal{T}_{H+1}=\mathcal{T}_0+(\mathcal{T}_{H+1}-\mathcal{T}_0)\leq
\mathcal{T}_0+C_T\ D_K
$$
with $C_T$ as in \equ{CT}, we obtain:
\beqano
C_{\sigma,H+1}&\leq&(\mathcal{T}_0+C_T\ D_K)
\Big\{|\lambda-1|\ \Big[{1\over {||\lambda|-1|}}(\|S_0\|_{\rho_0}+C_S D_K)+1\Big]\nonumber\\
&&+\Big(\|S_0\|_{\rho_0}+C_S D_K\Big)\Big\}\ \Big(\|M_0^{-1}\|_{\rho_0}+C_{Minv}D_K \Big)\nonumber\\
&=&\mathcal{T}_0\ \Big[|\lambda-1|\Big({1\over
{||\lambda|-1|}}\|S_0\|_{\rho_0}+1\Big)\
\|M_0^{-1}\|_{\rho_0}+\|S_0\|_{\rho_0}\|M_0^{-1}\|_{\rho_0}\Big]
+C_\sigma D_K\nonumber\\
&=&C_{\sigma 0}+C_\sigma D_K\nonumber\\
&\leq&2C_{\sigma0}\ ,
\eeqano
if
\beqa{Csigsig}
C_\sigma&\equiv&C_T\ \Big\{|\lambda-1|\ \Big[{1\over {||\lambda|-1|}}(\|S_0\|_{\rho_0}+C_S D_K)+1\Big]\nonumber\\
&&+\Big(\|S_0\|_{\rho_0}+C_S D_K\Big)\Big\}\ \Big(\|M_0^{-1}\|_{\rho_0}+C_{Minv}D_K \Big)\nonumber\\
&+&\mathcal{T}_0\ \Big\{|\lambda-1|\Big[{1\over {||\lambda|-1|}}(\|S_0\|_{\rho_0}+C_SD_K)+1\Big]C_{Minv}\nonumber\\
&+&|\lambda-1|\ {1\over {||\lambda|-1|}}\
\|M_0^{-1}\|_{\rho_0}\, C_S\nonumber\\
&+&C_S\Big(\|M_0^{-1}\|_{\rho_0}+C_{Minv}D_K \Big)+
C_{Minv}\|S_0\|_{\rho_0} \Big\}\ ,\nonumber\\
\eeqa
and if we require \equ{conde}.

\vskip.1in

\subsubsection{Proof of $C_{\E,H+1}\leq 2 C_{\E 0}$}

Recall that $\delta_{H+1}={{\delta_0}\over {2^{H+1}}}$ and that from \equ{Cepsp}, \equ{cr}, one has
$$
C_{\E,H+1}\equiv C_c\ C_{W,H+1}\nu \delta_{H+1}^{-1+\tau}+C_{\mathcal{R},H+1}\ .
$$

First, it suffices to prove that
\beq{CWH}
C_{W,H+1}\leq C_{W0}+C_W\ D_K\ ,
\eeq
for a suitable constant $C_W$ which will be given in \equ{CW} below.

\vskip.1in

From \equ{Csigmah}, for $C_{W_2,H+1}$ we have:
\beqano
C_{W_2,H+1}&\leq&{1\over {||\lambda|-1|}}\ \ \Big[1+2Q_{\mu 0}(C_{\sigma0}+D_KC_\sigma)\Big]
(\|M_0^{-1}\|_{\rho_0}+D_K)\nonumber\\
&\leq&C_{W_2 0}+D_K\ C_{W_2}\ ,
\eeqano
where
\beq{CW2}
C_{W_2}\equiv{1\over {||\lambda|-1|}}\ \Big[1+2Q_{\mu 0} \|M_0^{-1}\|_{\rho_0}C_\sigma+
2Q_{\mu 0} C_{\sigma0}+2Q_{\mu 0}C_\sigma D_K\Big]\ .
\eeq
Concerning $\overline{C}_{W_2,H+1}$, we have:
\beqano
\overline{C}_{W_2,H+1}&\leq&4(\mathcal{T}_0+C_T D_K)\ \Big[{1\over {||\lambda|-1|}}
(\|S_0\|_{\rho_0}+C_SD_K)+1\Big]\ Q_{\mu 0}\ (\|M_0^{-1}\|_{\rho_0}+D_K)^2\nonumber\\
&=&\overline{C}_{W_2 0}+\overline{C}_{W_2}D_K
\eeqano
with
\beqa{CW2new}
\overline{C}_{W_2}&\equiv& 4C_T\ \Big[{1\over {||\lambda|-1|}}(\|S_0\|_{\rho_0}+C_SD_K)+1\Big]\ Q_{\mu 0}
(\|M_0^{-1}\|_{\rho_0}+D_K)^2\nonumber\\
&+&4\mathcal{T}_0 Q_{\mu 0}\
{1\over {||\lambda|-1|}}C_S\ (\|M_0^{-1}\|_{\rho_0}+D_K)^2\nonumber\\
&+&4\mathcal{T}_0\ Q_{\mu 0} \Big[{1\over {||\lambda|-1|}} (\|S_0\|_{\rho_0}+C_SD_K)+1\Big]
(D_K+2\|M_0^{-1}\|_{\rho_0})\ .\nonumber\\
\eeqa
As for $C_{W_1,H+1}$ we have:
\beqano
C_{W_1,H+1}&\leq&C_0 \
\Big[(\|S_0\|_{\rho_0}+C_SD_K)(C_{W_20}+C_{W_2}D_K+\overline{C}_{W_2 0}+\overline{C}_{W_2} D_K)\nonumber\\
&+&\|M_0^{-1}\|_{\rho_0}+D_K+2Q_{\mu 0}(\|M_0^{-1}\|_{\rho_0}+D_K)(C_{\sigma 0}+D_KC_\sigma)\Big]\nonumber\\
&=&C_{W_1 0}+D_KC_{W_1}\ ,
\eeqano
where
\beqa{CW1}
C_{W_1}&\equiv&C_0\Big[\|S_0\|_{\rho_0}C_{W_2}+C_S C_{W_2 0}+C_S C_{W_2}D_K+\|S_0\|_{\rho_0} \overline{C}_{W_2} +C_S \overline{C}_{W_2 0}+
C_S \overline{C}_{W_2} D_K+1\nonumber\\
&+&2Q_{\mu 0}\|M_0^{-1}\|_{\rho_0}C_\sigma+2Q_{\mu 0} C_{\sigma 0}+
2Q_{\mu 0}C_\sigma D_K\Big]\ .
\eeqa
In conclusion, from \equ{Csigmah} we have:
\beqano
C_{W,H+1}&\equiv& (C_{W_1 0}+D_KC_{W_1})+(C_{W_2 0}+D_KC_{W_2}+\overline{C}_{W_2 0}+D_K \overline{C}_{W_2})
\nu\delta_0^\tau 2^{-\tau(H+1)}\nonumber\\
&\leq& C_{W0}+C_W D_K
\eeqano
with
\beq{CW}
C_W\equiv C_{W_1}+C_{W_2}\nu\delta_0^{\tau}+\overline{C}_{W_2}\nu\delta_0^{\tau}\ .
\eeq

\vskip.1in

In order to get $C_{\E,H+1}$ as in \equ{Cepsp} we estimate $C_{\mathcal{R},H+1}$. To this end, we
use the following inequality:
\beq{QEH}
Q_{E,H+1}\leq Q_{E0}+C_Q D_{2K}\ ,
\eeq
for a suitable constant $C_Q$ that will be given later in \equ{CQ} and for $D_{2K}$ defined as
\beq{D2K}
D_{2K}\equiv 4 C_{d0}\ C_c^2\ \nu^{-1} \delta_0^{-\tau-2}\ \varepsilon_0\ .
\eeq
We postpone for a moment the proof of \equ{QEH} and we rather stress that,
as a consequence of \equ{QEH}, we obtain:
\beqano
C_{\mathcal{R},H+1}&\leq&
Q_{E,H+1}\
(\|M_{H+1}\|_{\rho_{H+1}}^2\,C_{W,H+1}^2+C_{\sigma,H+1}^2\nu^{2}
\delta_{H+1}^{2\tau})\nonumber\\
&\leq& (Q_{E 0}+C_QD_{2K})\
\Big[(\|M_0\|_{\rho_0}+C_MD_K)^2
(C_{W0}+C_W D_K)^2\nonumber\\
&+&(C_{\sigma 0}+C_\sigma D_K)^2\nu^2\delta_0^{2\tau}\ 2^{-2\tau(H+1)}\Big]\nonumber\\
&\leq&C_{\mathcal{R} 0}+C_{\mathcal{R}}D_K\ ,
\eeqano
where
\beqa{CRR}
C_{\mathcal{R}}&\equiv&Q_{E0}\ \Big[(2C_M\|M_0\|_{\rho_0}+C_M^2D_K)
(C_{W0}+C_W D_K)^2+\|M_0\|_{\rho_0}^2(C_W^2D_K+2C_{W0}C_W)\nonumber\\
&+&(C_\sigma^2D_K+2C_{\sigma 0}C_\sigma)\nu^2\delta_0^{2\tau}\Big]
+C_Q\ \Big[(\|M_0\|_{\rho_0}+C_M D_K)^2(C_{W0}+C_W D_K)^2\nonumber\\
&+&(C_{\sigma 0}+C_\sigma D_K)^2\nu^2\delta_0^{2\tau}\Big]\
C_c\delta_0^{-1}
\eeqa
with $D_{2K}$ is as in \equ{D2K}.

We obtain that
\beqano
C_{\E,H+1}&\leq&(C_{W0}+C_W D_K)C_c\ \nu \delta_0^{-1+\tau}2^{-(-1+\tau)(H+1)}
+C_{\mathcal{R} 0}+C_{\mathcal{R}}D_K\nonumber\\
&\leq&C_{W0} C_c \nu\delta_0^{-1+\tau}+C_{\mathcal{R} 0}+ D_K\ \Big(C_W C_c
\nu \delta_0^{-1+\tau}+C_{\mathcal{R}}\Big)\nonumber\\
&\leq&2C_{\E 0}\ ,
\eeqano
if \equ{condf} is satisfied.

Let us conclude by proving \equ{QEH} starting from the definition
$$
Q_{E,H+1}\equiv {1\over
2}\max\Big\{\|D^2E_{H+1}\|_{\rho_{H+1}- \delta_{H+1}},\|DD_\mu
E_{H+1}\|_{\rho_{H+1}- \delta_{H+1}}, \|D^2_\mu E_{H+1}\|_{\rho_{H+1}- \delta_{H+1}}\Big\}\ .
$$
We recall that
$$
E_{H+1}=\E[K_{H+1},\mu_{H+1}]  = f_{\mu_{H+1}} \circ K_{H+1} -
K_{H+1} \circ T_\omega\ .
$$
It is convenient to introduce $\Delta_H$ and $\Xi_H$ such that
$$
K_{H+1}=K_0+(K_{H+1}-K_0)\equiv K_0+\Delta_H\ ,\qquad \mu_{H+1}=\mu_0+\sum_{j=0}^H\
\sigma_j\equiv \mu_0+\Xi_H\ .
$$
Then, we have the following bound on $E_{H+1}$:
\vspace{-10pt}
\beqano
\|E_{H+1}\|_{\rho_{H+1}- \delta_{H+1}}&=&\|(f_{\mu_0}\circ K_0-K_0\circ T_\omega)+f_{\mu_{H+1}}\circ K_{H+1}
-f_{\mu_0}\circ K_0\nonumber\\
&-&(K_{H+1}-K_0)\circ T_\omega\|_{\rho_{H+1}- \delta_{H+1}}\nonumber\\
&\leq&\|E_0\|_{\rho_0}\ +(1+\sup_{z\in\mathcal{C},\mu\in\Lambda,|\mu-\mu_0|<2\kappa_\mu\varepsilon_0}|Df_{\mu_0}(z)|)\ \|K_{H+1}-K_0\|_{\rho_{H+1}}\nonumber\\
&+&\sup_{z\in\mathcal{C},\mu\in\Lambda,|\mu-\mu_0|<2\kappa_\mu\varepsilon_0}|D_\mu f_{\mu_0}(z)|\ |\mu_{H+1}-\mu_0|\nonumber\\
&\leq&\|E_0\|_{\rho_0}\ +(1+\sup_{z\in\mathcal{C},\mu\in\Lambda,|\mu-\mu_0|<2\kappa_\mu\varepsilon_0}|Df_{\mu}(z)|)\
\kappa_K\varepsilon_0\nonumber\\
&+& \sup_{z\in\mathcal{C},\mu\in\Lambda,|\mu-\mu_0|<2\kappa_\mu\varepsilon_0}|D_\mu f_{\mu}(z)|\
\kappa_\mu\varepsilon_0\ ,
\eeqano
where we used \equ{Kmu}.

We now observe that the derivative of $f\circ K$ is given by
\beq{Dfk}
D(f\circ K)=D(f(K(\theta)))=Df(K(\theta))\ DK(\theta)
\eeq
and that the second derivative is given by
$$
D^2(f\circ K)=D^2(f(K(\theta)))(DK(\theta))^2+Df(K(\theta))\
D^2K(\theta)\ .
$$
Then, one has
\beqa{Q1}
\|D^2 E_{H+1}\|_{\rho_{H+1}}&\leq &\|D^2 E_0\|_{\rho_0}+D_{2K}\nonumber\\
&\qquad +&\|D^2 f_{\mu_0+\Xi_H}(K_0+\Delta_H)\ (DK_0+D\Delta_H)\nonumber\\
&\qquad -&D^2 f_{\mu_0}(K_0)\ DK_0\|_{\rho_{H+1}}\ \|DK_0\|_{\rho_0}\nonumber\\
&\qquad +&\|D f_{\mu_0+\Xi_H}(K_0+\Delta_H)-D f_{\mu_0}(K_0)\|_{\rho_{H+1}}\ \|D^2 K_0\|_{\rho_0}\nonumber\\
&\qquad +&\|D^2 f_{\mu_0+\Xi_H}(K_0+\Delta_H)\ (DK_0+D\Delta_H)\ D\Delta_H\|_{\rho_{H+1}}\nonumber\\
&\qquad +&\|D f_{\mu_0+\Xi_H}(K_0+\Delta_H)\ D^2 \Delta_H\|_{\rho_{H+1}}\nonumber\\
&\leq&\|D^2 E_0\|_{\rho_0}+D_{2K}+\sup_{z\in\mathcal{C}}|D^3 f_{\mu_0}(z)|\ \|DK_0\|_{\rho_0}^2\ \kappa_K\varepsilon_0\nonumber\\
&\qquad +&\sup_{z\in\mathcal{C},\mu\in\Lambda,|\mu-\mu_0|<2\kappa_\mu\varepsilon_0}|D_\mu D^2 f_{\mu}(z)|\ \|DK_0\|_{\rho_0}^2\ \kappa_\mu\varepsilon_0\nonumber\\
&\qquad +&\sup_{z\in\mathcal{C}}|D^2 f_{\mu_0}(z)|\ \|DK_0\|_{\rho_0}\ D_K\nonumber\\
&\qquad +&\sup_{z\in\mathcal{C}}|D^3 f_{\mu_0}(z)|\ \|DK_0\|_{\rho_0}\ \kappa_K\varepsilon_0\ D_K\nonumber\\
&\qquad +&\sup_{z\in\mathcal{C},\mu\in\Lambda,|\mu-\mu_0|<2\kappa_\mu\varepsilon_0}|D_\mu D^2 f_{\mu}(z)|\ \|DK_0\|_{\rho_0}\ \kappa_\mu\varepsilon_0\ D_K\nonumber\\
&\qquad +&\sup_{z\in\mathcal{C}}|D^2 f_{\mu_0}(z)|\ \|D^2 K_0\|_{\rho_0}^2\ \kappa_K\varepsilon_0\nonumber\\
&\qquad +&\sup_{z\in\mathcal{C},\mu\in\Lambda,|\mu-\mu_0|<2\kappa_\mu\varepsilon_0}|D_\mu D f_{\mu}(z)|\ \|D^2K_0\|_{\rho_0}\ \kappa_\mu\varepsilon_0\nonumber\\
&\qquad +&\sup_{z\in\mathcal{C}}|D^2 f_{\mu_0}(z)|\ (\|DK_0\|_{\rho_0}+D_K)\ D_K\nonumber\\
&\qquad +&\sup_{z\in\mathcal{C}}|D^3 f_{\mu_0}(z)|\ (\|DK_0\|_{\rho_0}+D_K)D_K\ \kappa_K\varepsilon_0\nonumber\\
&\qquad +&\sup_{z\in\mathcal{C},\mu\in\Lambda,|\mu-\mu_0|<2\kappa_\mu\varepsilon_0}|D_\mu D^2 f_{\mu}(z)|\ (\|DK_0\|_{\rho_0}+D_K)D_K\ \kappa_\mu\varepsilon_0\nonumber\\
&\qquad +&\sup_{z\in\mathcal{C}}|D f_{\mu_0}(z)|\ D_{2K}+\sup_{z\in\mathcal{C}}|D^2 f_{\mu_0}(z)|\ \kappa_K\varepsilon_0\ D_{2K}\nonumber\\
&\qquad +&\sup_{z\in\mathcal{C},\mu\in\Lambda,|\mu-\mu_0|<2\kappa_\mu\varepsilon_0}|D_\mu D f_{\mu}(z)|\ \kappa_\mu\varepsilon_0\ D_{2K}\ ,
\eeqa
with $\|DK_H\|_{\rho_H}\leq\|DK_0\|_{\rho_0}+D_K$, where $D_K$ was estimated as in \equ{DK1},
$\|D^2 K_H\|_{\rho_H}\leq\|D^2K_0\|_{\rho_0}+D_{2K}$, where $D_{2K}$ is defined through the following inequalities:
\beqano
\|D^2 K_{H+1}-D^2K_0\|_{\rho_{H+1}}&\leq&
\sum_{j=1}^H\|D^2\Delta_j\|_{\rho_j}\leq
C_c\ \sum_{j=1}^H\delta_j^{-1}v_j\nonumber\\
&\leq&C_c^2\ \sum_{j=1}^H C_{dj}\nu^{-1}\delta_j^{-\tau-2}\varepsilon_j\leq
4C_c^2\ C_{d0}\nu^{-1}\delta_0^{-\tau-2}\varepsilon_0 = D_{2K}\ ,
\eeqano
if \equ{conda} holds.

In a similar way we obtain the following estimate.
Given $f\circ K$, from \equ{Dfk} we have
$$
DD_\mu(f\circ K)=DD_\mu(f(K(\theta))) DK(\theta)\ .
$$
Then, we have
$$
DD_\mu E_{H+1}=DD_\mu E_0+DD_\mu f_{\mu_0}(K_0)\
(DK_{H+1}-DK_0)+DD_\mu(f_{\mu_{H+1}}(K_0+\Delta_H)-f_{\mu_0}(K_0))\ DK_{H+1}\ ,
$$
so that
\beqa{Q2}
\|DD_\mu E_{H+1}\|_{\rho_{H+1}}&\leq &\|DD_\mu E_0\|_{\rho_0}+\sup_{z\in\mathcal{C}}|DD_\mu f_{\mu_0}(z)|\ D_K\nonumber\\
&+&\sup_{z\in\mathcal{C}}|D^2 D_\mu f_{\mu_0}(z)|\ \|\Delta_H\|_{\rho_{H+1}}
\ (\|DK_0\|_{\rho_0}+\|D\Delta_H\|_{\rho_{H+1}})\nonumber\\
&+&\sup_{z\in\mathcal{C},\mu\in\Lambda,|\mu-\mu_0|<2\kappa_\mu\varepsilon_0}
|DD_\mu^2 f_{\mu_0}(z)|\ \|\Xi_H\|_{\rho_{H+1}}(\|DK_0\|_{\rho_0}+\|D\Delta_H\|_{\rho_{H+1}})\nonumber\\
&\leq&\|DD_\mu E_0\|_{\rho_0}+\sup_{z\in\mathcal{C}}|DD_\mu f_{\mu_0}(z)|\ D_K\nonumber\\
&+&\sup_{z\in\mathcal{C}}|D^2 D_\mu f_{\mu_0}(z)|\ \kappa_K\varepsilon_0\ (\|DK_0\|_{\rho_0}+D_K)\nonumber\\
&+&\sup_{z\in\mathcal{C},\mu\in\Lambda,|\mu-\mu_0|<2\kappa_\mu\varepsilon_0}|DD_\mu^2 f_{\mu}(z)|\ \kappa_\mu\varepsilon_0\
(\|DK_0\|_{\rho_0}+D_K)\ .
\eeqa
Finally, we have:
\beq{Q3}
\|D_\mu^2 E_{H+1}\|_{\rho_{H+1}}\leq\|D_\mu^2
E_0\|_{\rho_0}+\sup_{z\in\mathcal{C},\mu\in\Lambda,|\mu-\mu_0|<2\kappa_\mu\varepsilon_0}|D_\mu^3 f_{\mu}(z)|\
\kappa_\mu\varepsilon_0\ .
\eeq
Casting together \equ{Q1}, \equ{Q2}, \equ{Q3}, we obtain:
$$
Q_{E,H+1}\leq Q_{E0}+C_Q\ D_{2K}
$$
with
\beqa{CQ}
C_Q&\equiv&{1\over 2}\
\max\Big\{1+\sup_{z\in\mathcal{C}}|D^3 f_{\mu_0}(z)|\ \|DK_0\|_{\rho_0}^2\ C_c^{-2}\delta_0^2\nonumber\\
&+&\sup_{z\in\mathcal{C},\mu\in\Lambda,|\mu-\mu_0|<2\kappa_\mu\varepsilon_0}|D_\mu D^2 f_{\mu}(z)|\ \|DK_0\|_{\rho_0}^2\
{{C_{\sigma 0}}\over {C_{d0}}}C_c^{-2}\delta_0^{\tau+2}\nonumber\\
&+&\sup_{z\in\mathcal{C}}|D^2 f_{\mu_0}(z)|\ \|DK_0\|_{\rho_0}\ C_c^{-1}\delta_0\nonumber\\
&+&\sup_{z\in\mathcal{C}}|D^3 f_{\mu_0}(z)|\ \|DK_0\|_{\rho_0}\ 4 C_{d0} C_c^{-1}\nu^{-1}\delta_0^{-\tau+1}\varepsilon_0\nonumber\\
&+&\sup_{z\in\mathcal{C},\mu\in\Lambda,|\mu-\mu_0|<2\kappa_\mu\varepsilon_0}|D_\mu D^2 f_{\mu}(z)|\ \|DK_0\|_{\rho_0}\
4 C_{\sigma 0} C_c^{-1}\delta_0\varepsilon_0\nonumber\\
&+&\sup_{z\in\mathcal{C}}|D^2 f_{\mu_0}(z)|\ \|D^2 K_0\|_{\rho_0}^2\ C_c^{-2}\delta_0^2\nonumber\\
&+&\sup_{z\in\mathcal{C},\mu\in\Lambda,|\mu-\mu_0|<2\kappa_\mu\varepsilon_0}|D_\mu D f_{\mu}(z)|\ \|D^2K_0\|_{\rho_0}\
{{C_{\sigma 0}}\over {C_{d0}}}C_c^{-2}\nu\delta_0^{\tau+2}\nonumber\\
&+&\sup_{z\in\mathcal{C}}|D^2 f_{\mu_0}(z)|\ (\|DK_0\|_{\rho_0}+D_K)\ C_c^{-1}\delta_0\nonumber\\
&+&\sup_{z\in\mathcal{C}}|D^3 f_{\mu_0}(z)|\ (\|DK_0\|_{\rho_0}+D_K)4 C_{d0}\ C_c^{-1}\nu^{-1}\delta_0^{-\tau+1}\varepsilon_0\nonumber\\
&+&\sup_{z\in\mathcal{C},\mu\in\Lambda,|\mu-\mu_0|<2\kappa_\mu\varepsilon_0}|D_\mu D^2 f_{\mu}(z)|\ (\|DK_0\|_{\rho_0}+D_K)\
4C_{\sigma 0} C_c^{-1} \delta_0\varepsilon_0\nonumber\\
&+&\sup_{z\in\mathcal{C}}|D f_{\mu_0}(z)|+\sup_{z\in\mathcal{C}}|D^2 f_{\mu_0}(z)|\ \kappa_K\varepsilon_0\nonumber\\
&+&\sup_{z\in\mathcal{C},\mu\in\Lambda,|\mu-\mu_0|<2\kappa_\mu\varepsilon_0}|D_\mu D f_{\mu}(z)|\ \kappa_\mu\varepsilon_0\ ,\nonumber\\
&&\sup_{z\in\mathcal{C}}|DD_\mu f_{\mu_0}(z)|\ C_c^{-1}\delta_0+
\sup_{z\in\mathcal{C}}|D^2 D_\mu f_{\mu_0}(z)|\ C_c^{-2}\delta_0^2\ (\|DK_0\|_{\rho_0}+D_K)\nonumber\\
&+&\sup_{z\in\mathcal{C},\mu\in\Lambda,|\mu-\mu_0|<2\kappa_\mu\varepsilon_0}|DD_\mu^2 f_{\mu}(z)|\
{{C_{\sigma 0}}\over {C_{d0}}}C_c^{-2}\nu\delta_0^{\tau+2}\ (\|DK_0\|_{\rho_0}+D_K),\nonumber\\
&&\sup_{z\in\mathcal{C},\mu\in\Lambda,|\mu-\mu_0|<2\kappa_\mu\varepsilon_0}|D_\mu^3 f_{\mu}(z)|\
{{C_{\sigma 0}}\over {C_{d0}}}C_c^{-2}\nu\delta_0^{\tau+2}\Big\}\ .
\eeqa

\subsubsection{Proof of $C_{d,H+1}\leq 2C_{d0}$}
From \equ{Chath}, \equ{CWH} we have:
\beqano
C_{d,H+1}&=&\|M_{H+1}\|_{\rho_{H+1}}\ C_{W,H+1}\leq (\|M_0\|_{\rho_0}+D_K)(C_{W0}+C_WD_K)\nonumber\\
&\leq&C_{d0}+D_K\ (C_{W0}+\|M_0\|_{\rho_0}C_W+C_WD_K)\nonumber\\
&\leq&2C_{d0}\ ,
\eeqano
if \equ{condg} is satisfied.

\section{KAM estimates for the standard map}\label{sec:results}

In this Section we implement Theorem~\ref{main} to obtain explicit estimates on
the numerical validation of the golden mean curve of the dissipative standard map \equ{dsm}
that are close to the numerical breakdown value.
As mentioned in Section~\ref{sec:sketch}, we need to start with an approximate solution
$(K_0,\mu_0)$, which satisfies the invariance equation \equ{invariance} with an error
term $E_0$, whose norm on a domain of radius $\rho_0>0$ was denoted as $\varepsilon_0$
in Theorem~\ref{main}.

The construction of the approximate solution $(K_0,\mu_0)$ can be obtained by implementing
the algorithm described in \cite{CallejaCL11} and reviewed in Section~\ref{sec:approximate} below.
An estimate on the quantity $\varepsilon_0$ is obtained by imposing the list of conditions
\equ{C1}-\equ{C10}; explicit bounds are given in
Section~\ref{sec:value}, using the definitions of the constants provided in Appendix~\ref{app:constants}.


\subsection{Construction of the approximate solution}\label{sec:approximate}

To construct an approximate solution $(K_0,\mu_0)$ of the invariance equation \equ{invariance},
we make use of the fact that the a-posteriori format described in \cite{CallejaCL11} provides
an explicit algorithm, which can be implemented numerically in a very efficient way.
Each step of the algorithm is denoted as follows:
``$a\gets b$" means that the quantity $a$ is assigned by the quantity $b$.

\begin{algorithm}\label{alg:step}
Given $K_0: \torus \to \M$, $\mu_0 \in \real$,
we denote by   $\lambda\in \real$
the conformal factor for $f_{\mu_0}$. We perform the following computations:
\begin{itemize}
\item[1)] $E_0\gets f_{\mu_0} \circ K_0 - K_0\circ T_\omega$
\item[2)] $\alpha \gets DK_0$
\item[3)] $N_0\gets [\alpha^\top \alpha]^{-1}$
\item[4)] $M_0\gets [\alpha|\ J^{-1}\circ K_0\alpha N_0]$
\item[5)] $\beta\gets M_0^{-1}\circ T_\omega$
\item[6)] $\widetilde E_0\gets \beta E_0$
\item[7)] $P_0\gets\alpha N_0$
\item[{}] $S_0\gets (P_0\circ T_\omega)^\top Df_{\mu_0} \circ K_0\ J^{-1}\circ K_0 P_0$
\item[{}] $\widetilde A_0\gets M_0^{-1}\circ T_\omega\ D_\mu f_{\mu_0} \circ K_0$
\item[8)] $(B_{a0})^0$ solves \quad
$\lambda (B_{a0})^0 - (B_{a0})^0 \circ T_\omega = - (\widetilde E_0^{(2)})^0$
\item[{}] $(B_{b0})^0$ solves \quad
$\lambda (B_{b0})^0 - (B_{b0})^0 \circ T_\omega = - (\widetilde A_0^{(2)})^0$
\item[9)] Find $\overline{W}_0^{(2)}$, $\sigma_0$ solving
\beqano
&&0 = - \overline{S}_0\, \overline{W}_0^{(2)}
- \overline{ S_0 (B_{a0})^0   }
- \overline{ S_0 (B_{b0})^0   } \sigma_0
 - \overline{{\widetilde E}_0^{(1)}}
 - \overline{\widetilde A_0^{(1)}}\sigma_0\\
&&(\lambda -1) \overline{W}_0^{(2)}
= -\overline{\widetilde{E}_0^{(2)}} - \overline{\widetilde A_0^{(2)} }\sigma_0\ .
\eeqano
\item[10)] $(W_0^{(2)})^0=(B_{a0})^0+\sigma_0 (B_{b0})^0$
\item[11)]  $W_0^{(2)}=(W_0^{(2)})^0+\overline{W}_0^{(2)}$
\item[12)] $(W_0^{(1)})^0$ solves $\noaverage{W_0^{(1)}} -
\noaverage{W_0^{(1)}}\circ T_\omega = - \noaverage{S_0 W_0^{(2)}}
 -\noaverage { \widetilde E_0^{(1)}} - \noaverage{\widetilde A_0^{(1)}} \sigma_0$
\item[13)] $K_1 \gets K_0 + M_0 W_0$
\item[{}] $\mu_0 \gets \mu_0 +\sigma_0$\ .
\end{itemize}
\end{algorithm}

\begin{remark}
We call attention on the fact that steps 2), 8), 10), 11), 12)
involve diagonal operations in the Fourier space. On the contrary,
the other steps are diagonal in the real space
(while steps 10), 11) are diagonal in both spaces).
If we represent a function in discrete points or
in Fourier space, then we can compute the other functions by applying
the Fast Fourier Transform (FFT). This implies that if we use $N$ Fourier modes
to discretize the function, then we need $O(N)$ storage and $O(N\log N)$ operations.
\end{remark}

Next task is to translate the procedure described before into a
numerical algorithm that computes invariant tori of \equ{dsm}.
To this end, we fix the frequency equal to the golden ratio:
\beq{golden}
\omega={{\sqrt{5}-1}\over 2}\ .
\eeq
We remark that the golden ratio \equ{golden} satisfies the Diophantine condition \equ{DC}
with constants $\nu={2\over {3+\sqrt{5}}}$, $\tau=1$.

Then, we start from $(K_0,\mu_0)=(0,0)$, implement Algorithm~\ref{alg:step}
using Fast Fourier Transforms
and perform a continuation method to get an approximation of the invariant
circle close to the breakdown value.

\vskip.1in

To get closer to breakdown, one needs to implement Algorithm~\ref{alg:step}
with a sufficient accuracy.
The result described in Section~\ref{sec:value} is obtained
making all computations
by means of the GNU MPFR Library using $115$ significant digits.
We use our own extended precision implementation of the
classical radix-$2$ Cooley-Tukey
in \cite{CooleyTukey}
by using GNU MPFR.
We compute $2^{18}$ Fourier coefficients to discretize the invariant circle and
we ask for a tolerance equal to $10^{-46}$
in the approximation of the analytic norm \equ{normv},
of the invariance equation \equ{invariance}
to have convergence.

We fix $\lambda=0.9$ and (by trial and error to optimize the final result)
we select the parameters measuring the size of the domain as
$\rho_0=3\cdot 10^{-5}$, $\delta_0=\rho_0/4$. This choice of
$\rho_0$ is taken to optimize the final result.
We will denote by $\varepsilon_{KAM}$ the value of the parameter
$\varepsilon$ after the algorithm has converged to an approximate
solution, $(K, \mu)$, and all the estimates of Theorem~\ref{main}
(precisely \equ{C1}-\equ{C2}-\equ{C3}-\equ{C4}-\equ{condbT}-\equ{Cnew1}-\equ{Cnew2}-\equ{C8}-\equ{C9}-\equ{C10})
have been verified numerically for that approximate solution.
In fact, Table~\ref{tab:rho} provides the
value of $\varepsilon_{KAM}$
obtained with $2^{18}$ Fourier coefficients for different values
of $\rho_0$.
We emphasize that the a-posteriori format of Theorem~\ref{main} verifies the solution and does not need
to justify how the approximate solution is constructed.

\vskip.2in

\begin{table}[h]
\begin{tabular}{|c|c|c|c|}
  \hline
  $\rho_0$ & $\varepsilon_{KAM}$ & agreement with $\varepsilon_c$ & $\mu$ \\
  \hline
  $10^{-5}$ & 0.97094171 & 99.89$\%$ & 0.06139053\\
  $2\cdot 10^{-5}$ & 0.97136363 & 99.93$\%$ & 0.06139054\\
  $3\cdot 10^{-5}$ & 0.97142178 & 99.94$\%$ &0.06139056\\
  $4\cdot 10^{-5}$ & 0.97136363 & 99.93$\%$ &0.06139060\\
  $5\cdot 10^{-5}$ & 0.97133318 & 99.93$\%$ &0.06139063\\
  $6\cdot 10^{-5}$ & 0.97127502 & 99.92$\%$ &0.06139068\\
  $7\cdot 10^{-5}$ & 0.97120503 & 99.92$\%$ &0.06139072\\
  $8\cdot 10^{-5}$ & 0.97114973 & 99.91$\%$ &0.06139075\\
  $9\cdot 10^{-5}$ & 0.97094171 & 99.89$\%$ &0.06139079\\
  $10^{-4}$ & 0.97094171 & 99.89$\%$ &0.06139082\\
  $2\cdot 10^{-4}$ & 0.97011584 & 99.80$\%$ &0.06139146\\
  \hline
\end{tabular}
\vskip.1in
\caption{The analytical estimate $\varepsilon_{KAM}$ for the golden mean curve of \equ{dsm} with $\lambda=0.9$
    for different values of the parameter $\rho$ measuring the width of the analyticity domain
    considered for $K$.}\label{tab:rho}
\end{table}

\vskip.2in

As Table~\ref{tab:Fourier} shows, the higher the number of Fourier coefficients,
the better is the result, although the execution time becomes longer. We also notice that the improvement
is smaller as the number of Fourier coefficients increases; in particular, the results are very
similar when taking $2^{17}$ and $2^{18}$ Fourier coefficients.

\vskip.2in

\begin{table}[h]
\begin{tabular}{|c|c|c|c|c|}
  \hline
  n. Fourier  & $\varepsilon_{KAM}$ & $\mu$ & agreement & execution time \\
  coefficients &  &  & with $\varepsilon_c$ & (sec) \\
  \hline
  $2^{13}$ & 0.95730400 & 0.06140120 & 98.49$\%$ & 612.28 \\
  $2^{14}$ & 0.96512016 & 0.06139562 & 99.29$\%$ & 2015.22 \\
  $2^{15}$ & 0.96807778 & 0.06139307 & 99.60$\%$ & 3205.34 \\
  $2^{16}$ & 0.97011583 & 0.06139161 & 99.81$\%$ & 8460.19 \\
  $2^{17}$ & 0.97094171 & 0.06139089 & 99.89$\%$ & 13375.78 \\
  $2^{18}$ & 0.97142178 & 0.06139056 & 99.94$\%$ & 38222.48 \\
  \hline
\end{tabular}
\vskip.1in
\caption{The analytical estimate $\varepsilon_{KAM}$ for the golden mean curve of \equ{dsm} with $\lambda=0.9$, $\rho=3\cdot 10^{-5}$,
as the number of Fourier coefficients of the solution increases.}\label{tab:Fourier}
\end{table}

\vskip.2in

The output of the construction of the approximate solution via the MPRF program
is represented by the analytic norms of the following quantities, which will be used to
check the conditions \equ{C1}-\equ{C10}, needed to implement
Theorem~\ref{main}.
All quantities are given with 30 decimal digits:
\beqa{norms}
\|M_0\|_{\rho_0} &=&  44.9270811990274410452148184267\ ,\nonumber\\
\|M_0^{-1}\|_{\rho_0} &=& 39.930678840711850152808576113\ ,\nonumber\\
\|Df_{\mu_0}\|_{\rho_0} &=& 5.07550011737521959347639032433\ ,\nonumber\\
\|D^2 f_{\mu_0}\|_{\rho_0} &=& 12.2074077197778485732557018883\ ,\nonumber\\
\|S_0\|_{\rho_0} &=& 215.24720762912463716286404004\ ,\nonumber\\
\|N_0\|_{\rho_0} &=& 156.534312450915756580422752539\ ,\nonumber\\
\|N_0^{-1}\|_{\rho_0} &=& 591.408362768291837018626059244\ ,\nonumber\\
\|DK_0\|_{\rho_0} &=& 44.9270811990274410452148184267\ ,\nonumber\\
\|D^2 K_0\|_{\rho_0} &=& 221591.876024617607481468301961\ ,\nonumber\\
\|DK_0^{-1}\|_{\rho_0} &=& 7032.62976591622436294280767134\ ,\nonumber\\
{\mathcal T}_0 &=& 7.6434265622376167352649577512\ ,\nonumber\\
\|E_0\|_{\rho_0} &=& 7.71650351451832566847490849233\, 10^{-36}\ ,\nonumber\\
\|D^2 E_0\|_{\rho_0} &=& 5.1576300492851806964395530006\, 10^{-24}\ .
\eeqa
With reference to the quantities in \equ{QQQ}, we notice that in the case of the dissipative standard map
\equ{dsm} we have $Q_{\mu 0}=1$ and $Q_{z\mu 0}=Q_{\mu\mu 0}=0$.
We stress that the quantities which require the hardest computation effort is the error
$E_0$  and its derivatives.

\subsection{Check of the conditions of Theorem~\ref{main} and results}\label{sec:value}
We verify numerically the estimates of the theorem on the existence of the golden mean torus for
the dissipative standard map described by equation \equ{dsm} with frequency as in
\equ{golden} and $\lambda=0.9$.
The corresponding breakdown threshold, as computed by means of the Sobolev's
method used in \cite{CallejaC10},
or equivalently by means of Greene's technique (see \cite{CallejaC10},
\cite{CCFL14}), gives
\beq{epsc}
\varepsilon_c=0.97198\ ,
\eeq
(compare with \cite{CallejaC10}).

On the other hand, implementing the analytical estimates of
Section~\ref{sec:proof}, we obtain that the conditions
\equ{C1}-\equ{C10}, appearing in Theorem~\ref{main} are satisfied
for a value of the perturbing parameter equal to \beq{epsKAM}
\varepsilon_{KAM} = 0.971421780429401935547661013138\ . \eeq The
corresponding value of the drift parameter amounts to \beq{muKAM}
\mu = 0.061390559555891469231218991051\ . \eeq The result is
validated by running the program with different precision on a
DELL Machine with an Intel Xeon Processor E5-2643 (Quad Core, 3.30GHz
Turbo, 10MB, 8.0 GT/s) and 16GB RAM. Precisely, we provide in
Table~\ref{tab:res} the results with different significant digits.

\vskip.2in

\begin{table}[h]
\begin{tabular}{|c|c|c|}
  \hline
  digits & $\varepsilon_{KAM}$ & execution time (sec)\\
  \hline
  50 & 0.97142178 & 27632.88\\
  60 & 0.97142178 & 29027.68 \\
  70 & 0.97142178 & 30094.44 \\
  85 & 0.97142178 & 32685.89  \\
  100 & 0.97142178 & 35390.35 \\
  115 & 0.97142178 & 38222.48\\
  \hline
\end{tabular}
\vskip.1in
\caption{The analytical estimate $\varepsilon_{KAM}$ for the golden mean curve of \equ{dsm} with $\lambda=0.9$, $\rho=3\cdot 10^{-5}$,
number of Fourier coefficients equal to $2^{18}$ and
for different precision of the computation, obtained varying the number of digits as in the first column.}\label{tab:res}
\end{table}

\vskip.2in

The results shown in Table~\ref{tab:res} suggest that the norms provided in \equ{norms}
are robust and, even if we do not implement interval arithmetic, we can conjecture that the values provided
in \equ{norms} are not affected by numerical errors. Below 50 digits of precision, the algorithm does not
produce any result, since some quantities are so small that a precision less than 50 digits is not enough.
This remark leads us to state the following result.

\begin{theorem}\label{thm:result}
Let us consider the map \equ{dsm} with $\lambda=0.9$. Let
$\rho_0=3\cdot 10^{-5}$, $\delta_0=\rho_0/4$, $\zeta=3\cdot
10^{-5}$; let us fix the frequency as $\omega={{\sqrt{5}-1}\over
2}$. Assume that the norms of $M_0$, $M_0^{-1}$, $Df_{\mu_0}$,
$D^2f_{\mu_0}$, $S_0$, $N_0$, $N_0^{-1}$, $DK_0$, $D^2K_0$,
$DK_0^{-1}$, $E_0$, $D^2E_0$, and that the twist constant
${\mathcal T}_0$ are given by the values provided in \equ{norms}.

Then, there exists an invariant attractor with frequency $\omega$ for
$\varepsilon \approx \varepsilon_{KAM}$ with
$\varepsilon_{KAM}$ as in \equ{epsKAM} and for a value of the drift parameter as in \equ{muKAM}.
\end{theorem}

\vskip.1in

The result stated in Theorem~\ref{thm:result} verifies the estimates for
$\varepsilon_{KAM}$ which is
consistent within $99.94\%$ of the numerical value $\varepsilon_c$
given in \equ{epsc}. This result
shows that, beside a world-wide recognized theoretical interest, KAM theory can also provide a constructive
effective algorithm to estimate the breakdown value with great accuracy.




\vfill\eject

\appendix

\section{Proof of Lemma \ref{neutral}}\label{app:lemmaproof}
In this appendix, we include the proof of Lemma~\ref{neutral}. In
the proof we follow the construction of
\cite{Russmann76a} to derive the constant $C_0$ in \equ{Cu}.

\begin{proof}
For the proof of the existence of the solution of \eqref{difference} we refer to \cite{CallejaCL11},
where the constant $C_0$ was not made explicit. Here, instead, we
provide an explicit estimate of $C_0$, which closely follows \cite{Russmann75}.
Let us expand $\varphi$ and $\eta$ in Fourier series as
$$
\varphi(\theta)=\sum_{k\in\integer} \hat\varphi_k e^{2\pi i k\theta}\ ,\qquad
\eta(\theta)=\sum_{k\in\integer} \hat\eta_k e^{2\pi i k\theta}\ ,
$$
where $\hat\varphi_k$, $\hat\eta_k$ denote the Fourier coefficients.
Then, equation \equ{difference} becomes
$$
\sum_{k\in\integer} \hat\varphi_k(e^{2\pi ik\omega}-\lambda)e^{2\pi ik\theta}=\sum_{k\in\integer}\hat\eta_k
e^{2\pi i k\theta}\ ,
$$
providing
$$
\hat\varphi_k={{\hat\eta_k}\over {e^{2\pi i k\omega}-\lambda}}\ .
$$
Adding the Fourier coefficients, one obtains:
$$
\varphi(\theta)=\sum_{k\in\integer}{{\hat\eta_k}\over {e^{2\pi i k\omega}-\lambda}}\ e^{2\pi i k\theta}\ .
$$
Let
$$
Z_k\equiv \min_{q\in\integer} |\omega\,k-q|\ ;
$$
we have the following inequality
\beqano
|e^{2\pi i k\omega}-\lambda|^2&=&(1-\lambda)^2\cos^2(\pi k\omega)+(1+\lambda)^2\sin^2(\pi k\omega) \nonumber\\
&\geq&(1+\lambda)^2\sin^2(\pi k\omega) \geq 4(1+\lambda)^2 Z_k^2\ ,
\eeqano
where the last inequality comes from noticing that $\sin(x)/x\geq 2/\pi$ for all $0<x<{\pi\over 2}$.

Therefore we obtain:
$$
|e^{2\pi i k\omega}-\lambda|\geq 2(1+\lambda) |Z_k|\geq 2(1+\lambda) \nu|k|^{-\tau}\ ,
$$
namely
$$
|e^{2\pi i k\omega}-\lambda|^{-1}\leq{1\over {2(1+\lambda)}}\nu^{-1}|k|^{\tau}\ .
$$
Finally, we have
\beqano
\|\varphi\|_{\rho-\delta}&\leq&\sum_{k\in\integer} |{\hat \eta_k}| e^{2\pi \rho|k|}\ {{e^{-2\pi \delta|k|}}\over
{|e^{2\pi ik\omega}-\lambda|}}\nonumber\\
&\leq& \sqrt{\sum_{k\in\integer} |{\hat \eta_k}|^2\, e^{4\pi \rho|k|}}\
\sqrt{\sum_{k\in\integer} {{e^{-4\pi \delta|k|}}\over {|e^{2\pi ik\omega}-\lambda|^2}}}\nonumber\\
&\leq& \|\eta\|_\rho\ \sqrt{F(\delta)}\ ,
\eeqano
where
$$
F(\delta)\equiv 4\ \sum_{k=1}^\infty {{e^{-4\pi \delta|k|}}\over {|e^{2\pi ik\omega}-\lambda|^2}}
$$
and where we used the estimate (see \cite{Russmann75})
$$
\sum_{k\in\integer} |{\hat \eta_k}|^2\, e^{4\pi \rho|k|} \leq 2\|\eta\|_\rho^2\ .
$$
Denoting by $\Gamma$ the Euler gamma function, using the estimates of \cite{Russmann76a}, one has that
$$
F(\delta)\leq {{\pi^2 \Gamma(2\tau+1)}\over {3\nu^2 (1+\lambda)^2(2\delta)^{2\tau} (2\pi)^{2\tau}}}\ ,
$$
which leads to \equ{estimate} with $C_0$ as in \equ{Cu}.
\end{proof}

\section{Constants of the KAM theorem}\label{app:constants}
The constants entering in the conditions\equ{C1}-\equ{C10} of Theorem~\ref{main} are defined through the
following (long) list.
For fast reference, before each constant we provide the label of the formula where the constant
was introduced.
We note that the constants are given in an explicit format and evaluating them requires only
a few lines of code.
\beqano
\equ{NU}\qquad C_{\sigma 0}&\equiv&\mathcal{T}_0\ \Big[|\lambda-1|\ \Big({1\over {||\lambda|-1|}}\|S_0\|_{\rho_0}+1\Big)
+\|S_0\|_{\rho_0}\Big]\ \|M_0^{-1}\|_{\rho_0}\ ,\nonumber\\
\equ{NU}\qquad C_{W_2 0}&\equiv&{1\over {||\lambda|-1|}}\Big(1+C_{\sigma 0} Q_{\mu 0}\Big)\|M_0^{-1}\|_{\rho_0}\ ,\nonumber\\
\equ{NU}\qquad \overline{C}_{W_2 0} &\equiv& 2\mathcal{T}_0\ \Big({1\over {||\lambda|-1|}}\|S_0\|_{\rho_0}+1\Big)\,
Q_{\mu 0}\ \|M_0^{-1}\|_{\rho_0}^2\ ,\nonumber\\
\equ{NU}\qquad C_{W_1 0}&\equiv&C_0\Big(\|S_0\|_{\rho_0}(C_{W_2 0}+\overline{C}_{W_2 0})+\|M_0^{-1}\|_{\rho_0}+Q_{\mu 0} \|M_0^{-1}\|_{\rho_0}C_{\sigma 0}\Big)
\ ,\nonumber\\
\equ{NU}\qquad C_{W0}&\equiv&C_{W_1 0}+(C_{W_2 0}+\overline{C}_{W_2 0})\nu\delta_0^{\tau}\ ,\nonumber\\
\equ{Ceta}\qquad C_{\eta 0}&\equiv& C_{W0}\|M_0\|_{\rho_0}+C_{\sigma 0}\nu\delta_0^{\tau}\ ,\nonumber\\
\equ{CR}\qquad C_{\mathcal{R} 0}&\equiv&Q_{E 0}(\|M_0\|_{\rho_0}^2 C_{W0}^2+C_{\sigma 0}^2\nu^2\delta_0^{2\tau})\ ,\nonumber\\
\equ{Ceps'}\qquad C_{\E 0}&\equiv& C_c\ C_{W0}\nu \delta_0^{-1+\tau}+C_{\mathcal{R} 0}\ ,\nonumber\\
\equ{Chath}\qquad C_{d0}&\equiv& C_{W0}\ \|M_0\|_{\rho_0}\ ,\nonumber\\
\equ{kappa}\qquad \kappa_0&\equiv&2^{2\tau+1}\, C_{\E 0} \nu^{-2}\delta_0^{-2\tau}\ ,\nonumber\\
\equ{kappa}\qquad \kappa_K&\equiv& 4 C_{d0}\ \nu^{-1} \delta_0^{-\tau}\ ,\nonumber\\
\equ{kappa}\qquad \kappa_\mu&\equiv&4 C_{\sigma 0}\ ,\nonumber\\
\equ{DK1}\qquad D_K&\equiv& 4C_{d0}\ C_c\ \nu^{-1} \delta_0^{-\tau-1}\ \varepsilon_0\ ,\nonumber\\
\equ{D2K}\qquad D_{2K}&\equiv& 4\ C_{d0} C_c^2 \nu^{-1} \delta_0^{-\tau-2}\ \varepsilon_0\ ,\nonumber\\
\equ{constants}\qquad C_N&\equiv&\|N_0\|_{\rho_0}^2\
{{2\|DK_0\|_{\rho_0}+D_K}\over {1-\|N_0\|_{\rho_0}D_K(2\|DK_0\|_{\rho_0}+D_K)}}\ ,\nonumber\\
\equ{constants}\qquad C_M&\equiv&1+J_e\Big[C_N(\|DK_0\|_{\rho_0}+D_K)+\|N_0\|_{\rho_0}\Big]\ ,\nonumber\\
\equ{constants}\qquad C_{Minv}&\equiv&C_{N}(\|DK_0\|_{\rho_0}+D_K)+\|N_0\|_{\rho_0}+J_e\ ,\nonumber\\
\equ{CtildeS}\qquad C_S&\equiv& 2 J_e Q_0 \
\Big\{(\|N_0\|_{\rho_0}+C_ND_K)\ \Big[D_K
(\|N_0\|_{\rho_0}+C_ND_K)\nonumber\\
&+&\|DK_0\|_{\rho_0}\|N_0\|_{\rho_0}+\|DK_0\|_{\rho_0}C_ND_K\Big]\nonumber\\
&+&C_N\|DK_0\|_{\rho_0} \Big[D_K
(\|N_0\|_{\rho_0}+C_ND_K)
+\|DK_0\|_{\rho_0}\|N_0\|_{\rho_0}+\|DK_0\|_{\rho_0}C_N D_K\Big]\nonumber\\
&+&\|N_0\|_{\rho_0}\|DK_0\|_{\rho_0}(\|N_0\|_{\rho_0}+C_ND_K)
+C_N \|N_0\|_{\rho_0} \|DK_0\|_{\rho_0}^2\Big\}\ ,\nonumber\\
\eeqano

\beqano
\equ{Ctilde2S}\qquad C_{SB}&\equiv& {1\over {||\lambda|-1|}}Q_{\mu 0}\|M_0^{-1}\|_{\rho_0}C_S+
2 J_e Q_0\ \|N_0\|_{\rho_0}^2\ \|DK_0\|_{\rho_0}^2{1\over {||\lambda|-1|}}\ C_{Minv}
\ Q_{\mu 0}\nonumber\\
&+&2C_S\ {1\over {||\lambda|-1|}}\ C_{Minv}\ Q_{\mu 0}\ D_K\ ,\nonumber\\
\equ{taue}\qquad C_\tau&\equiv&\max\Big\{C_S, C_{SB}+2C_{Minv} Q_{\mu 0}\Big\}\ D_K\ ,\nonumber\\
\equ{CT}\qquad C_T&\equiv& {{\mathcal{T}_0^2}\over {1-\mathcal{T}_0 C_\tau}}\
\max\Big\{C_S, C_{SB}+2C_{Minv} Q_{\mu 0}\Big\}\ ,\nonumber\\
\equ{Csigsig}\qquad C_\sigma&\equiv&C_T\ \Big\{|\lambda-1|\ \Big[{1\over {||\lambda|-1|}}(\|S_0\|_{\rho_0}+C_S
D_K)+1\Big]\nonumber\\
&+&\Big(\|S_0\|_{\rho_0}+C_S D_K\Big)\Big\}\ \Big(\|M_0^{-1}\|_{\rho_0}+C_{Minv}D_K\Big)\nonumber\\
&+&\mathcal{T}_0\ \Big\{|\lambda-1|\ \Big[{1\over {||\lambda|-1|}}(\|S_0\|_{\rho_0}+C_S D_K)+1\Big]C_{Minv}\nonumber\\
&+&|\lambda-1|\ {1\over {||\lambda|-1|}}\ \|M_0^{-1}\|_{\rho_0}C_S+C_S
\Big(\|M_0^{-1}\|_{\rho_0}+C_{Minv}D_K\Big)\nonumber\\
&+&C_{Minv}\|S_0\|_{\rho_0}\Big\}\ ,\nonumber\\
\equ{CW2new}\qquad \overline{C}_{W_2}&\equiv& 4C_T\ \Big[{1\over {||\lambda|-1|}}(\|S_0\|_{\rho_0}+C_SD_K)+1\Big]\ Q_{\mu 0}
(\|M_0^{-1}\|_{\rho_0}+D_K)^2\nonumber\\
&+&4\mathcal{T}_0 Q_{\mu 0}\
{1\over {||\lambda|-1|}}C_S\ (\|M_0^{-1}\|_{\rho_0}+D_K)^2\nonumber\\
&+&4\mathcal{T}_0\ Q_{\mu 0} \Big[{1\over {||\lambda|-1|}} (\|S_0\|_{\rho_0}+C_SD_K)+1\Big]
(D_K+2\|M_0^{-1}\|_{\rho_0})\nonumber\\
\equ{CRR}\qquad C_{\mathcal{R}}&\equiv&Q_{E0}\ \Big[(2C_M\|M_0\|_{\rho_0}+C_M^2D_K)
(C_{W0}+C_WD_K)^2+\|M_0\|_{\rho_0}^2(C_W^2D_K+2C_{W0}\ C_W)\nonumber\\
&+&(C_\sigma^2D_K+2C_{\sigma 0}C_\sigma)\nu^2\delta_0^{2\tau}\Big]
+C_Q\ \Big[(\|M_0\|_{\rho_0}+C_M D_K)^2(C_{W0}+C_W D_K)^2\nonumber\\
&+&(C_{\sigma 0}+C_\sigma D_K)^2\nu^2\delta_0^{2\tau}\Big]\ C_c\delta_0^{-1}\ ,\nonumber\\
\eeqano

\beqano
\equ{CW2}\qquad C_{W_2}&\equiv&{1\over {||\lambda|-1|}}\ \Big[1+2Q_{\mu 0} \|M_0^{-1}\|_{\rho_0}C_\sigma
+2Q_{\mu 0} C_{\sigma 0}+2Q_{\mu 0}C_\sigma D_K\Big]\ ,\nonumber\\
\equ{CW1}\qquad C_{W_1}&\equiv& C_0\Big[\|S_0\|_{\rho_0}C_{W_2}+C_S C_{W_2 0}+
C_S C_{W_2}D_K+
\|S_0\|_{\rho_0} \overline{C}_{W_2}\nonumber\\
&+&C_S \overline{C}_{W_2 0}+C_S \overline{C}_{W_2}D_K+1\nonumber\\
&+&2Q_{\mu 0}\|M_0^{-1}\|_{\rho_0}C_\sigma+2Q_{\mu 0} C_{\sigma 0}+
2Q_{\mu 0} C_\sigma D_K\Big]\ ,\nonumber\\
\equ{CW}\qquad C_W&\equiv& C_{W_1}+C_{W_2}\nu\delta_0^{\tau}
+\overline{C}_{W_2}\nu\delta_0^{\tau}\nonumber\\
\equ{CQ}C_Q&\equiv&{1\over 2}\
\max\Big\{1+\sup_{z\in\mathcal{C}}|D^3 f_{\mu_0}(z)|\ \|DK_0\|_{\rho_0}^2\ C_c^{-2}\delta_0^2\nonumber\\
&+&\sup_{z\in\mathcal{C},\mu\in\Lambda,|\mu-\mu_0|<2\kappa_\mu\varepsilon_0}|D_\mu D^2 f_{\mu}(z)|\ \|DK_0\|_{\rho_0}^2\
{{C_{\sigma 0}}\over {C_{d0}}}C_c^{-2}\delta_0^{\tau+2}\nonumber\\
&+&\sup_{z\in\mathcal{C}}|D^2 f_{\mu_0}(z)|\ \|DK_0\|_{\rho_0}\ C_c^{-1}\delta_0\nonumber\\
&+&\sup_{z\in\mathcal{C}}|D^3 f_{\mu_0}(z)|\ \|DK_0\|_{\rho_0}\ 4 C_{d0} C_c^{-1}\nu^{-1}\delta_0^{-\tau+1}\varepsilon_0\nonumber\\
&+&\sup_{z\in\mathcal{C},\mu\in\Lambda,|\mu-\mu_0|<2\kappa_\mu\varepsilon_0}|D_\mu D^2 f_{\mu}(z)|\ \|DK_0\|_{\rho_0}\
4 C_{\sigma 0} C_c^{-1}\delta_0\varepsilon_0\nonumber\\
&+&\sup_{z\in\mathcal{C}}|D^2 f_{\mu_0}(z)|\ \|D^2 K_0\|_{\rho_0}^2\ C_c^{-2}\delta_0^2\nonumber\\
&+&\sup_{z\in\mathcal{C},\mu\in\Lambda,|\mu-\mu_0|<2\kappa_\mu\varepsilon_0}|D_\mu D f_{\mu}(z)|\ \|D^2K_0\|_{\rho_0}\
{{C_{\sigma 0}}\over {C_{d0}}}C_c^{-2}\nu\delta_0^{\tau+2}\nonumber\\
&+&\sup_{z\in\mathcal{C}}|D^2 f_{\mu_0}(z)|\ (\|DK_0\|_{\rho_0}+D_K)\ C_c^{-1}\delta_0\nonumber\\
&+&\sup_{z\in\mathcal{C}}|D^3 f_{\mu_0}(z)|\ (\|DK_0\|_{\rho_0}+D_K)4 C_{d0}\ C_c^{-1}\nu^{-1}\delta_0^{-\tau+1}\varepsilon_0\nonumber\\
&+&\sup_{z\in\mathcal{C},\mu\in\Lambda,|\mu-\mu_0|<2\kappa_\mu\varepsilon_0}|D_\mu D^2 f_{\mu}(z)|\ (\|DK_0\|_{\rho_0}+D_K)\
4C_{\sigma 0} C_c^{-1} \delta_0\varepsilon_0\nonumber\\
&+&\sup_{z\in\mathcal{C}}|D f_{\mu_0}(z)|+\sup_{z\in\mathcal{C}}|D^2 f_{\mu_0}(z)|\ \kappa_K\varepsilon_0\nonumber\\
&+&\sup_{z\in\mathcal{C},\mu\in\Lambda,|\mu-\mu_0|<2\kappa_\mu\varepsilon_0}|D_\mu D f_{\mu}(z)|\ \kappa_\mu\varepsilon_0\ ,\nonumber\\
&&\sup_{z\in\mathcal{C}}|DD_\mu f_{\mu_0}(z)|\ C_c^{-1}\delta_0+
\sup_{z\in\mathcal{C}}|D^2 D_\mu f_{\mu_0}(z)|\ C_c^{-2}\delta_0^2\ (\|DK_0\|_{\rho_0}+D_K)\nonumber\\
&+&\sup_{z\in\mathcal{C},\mu\in\Lambda,|\mu-\mu_0|<2\kappa_\mu\varepsilon_0}|DD_\mu^2 f_{\mu}(z)|\
{{C_{\sigma 0}}\over {C_{d0}}}C_c^{-2}\nu\delta_0^{\tau+2}\ (\|DK_0\|_{\rho_0}+D_K),\nonumber\\
&&\sup_{z\in\mathcal{C},\mu\in\Lambda,|\mu-\mu_0|<2\kappa_\mu\varepsilon_0}|D_\mu^3 f_{\mu}(z)|\
{{C_{\sigma 0}}\over {C_{d0}}}C_c^{-2}\nu\delta_0^{\tau+2}\Big\}\ .\nonumber\\
\eeqano


\def\cprime{$'$} \def\cprime{$'$} \def\cprime{$'$} \def\cprime{$'$}

\end{document}